\documentclass[reqno,openright,11pt,twoside]{amsart}

\usepackage{ifthen}
\usepackage{amssymb, bm, xspace}

\usepackage{enumerate}

\usepackage[pagebackref=true, colorlinks=true, citecolor=blue]{hyperref}

\usepackage[off]{pdfsync}

\usepackage{mathtools}

\renewcommand{\mathbf}[1]{\bm{#1} \textbf{ *** Use bm instead of mathbf ***}}

\newcounter{HaveBBM} 

\ifthenelse{\value{HaveBBM}=1}{
    \usepackage{bbm}
    }{}

\newcounter{dtlForSubmission} 
\newcounter{dtlMarginComments} 
\newcounter{dtlSomeDetail} 
\newcounter{dtlFullDetails} 

\newcounter{DetailLevel} 

\newcommand{\DetailMarginNote}[1]{
    \ifthenelse{\value{DetailLevel}=\value{dtlMarginComments} \or \value{DetailLevel}>\value{dtlMarginComments}}
        {{\small #1}}{}
    }

\newcommand{\DetailSome}[1]{
    \ifthenelse{\value{DetailLevel}=\value{dtlSomeDetail} \or \value{DetailLevel}>\value{dtlSomeDetail}}
        {\begin{quote}{\small \textbf{Detailed compile only}: #1}\end{quote}}{}
    }

\newcommand{\DetailFull}[1]{
    \ifthenelse{\value{DetailLevel}=\value{dtlFullDetails} \or \value{DetailLevel}>\value{dtlFullDetails}}
        {{\small \textbf{Detailed compile only}: #1}}{}
    }

\newcommand{\NotDetailSome}[1]{
    \ifthenelse{\value{DetailLevel}=\value{dtlSomeDetail} \or \value{DetailLevel}>\value{dtlSomeDetail}}
        {}{#1}
    }

\newcommand{\NotDetailFull}[1]{
    \ifthenelse{\value{DetailLevel}=\value{dtlFullDetails} \or \value{DetailLevel}>\value{dtlFullDetails}}
        {}{#1}
    }

\newcommand{\DetailSomeElse}[2]{
    \ifthenelse{\value{DetailLevel}=\value{dtlSomeDetail} \or \value{DetailLevel}>\value{dtlSomeDetail}}
        {{\small \textbf{Detailed compile only}: #1}}{#2}
    }

\newcommand{\DetailFullElse}[2]{
    \ifthenelse{\value{DetailLevel}=\value{dtlFullDetails} \or \value{DetailLevel}>\value{dtlFullDetails}}
        {{\small \textbf{Detailed compile only}: #1}}{#2}
    }

\newcommand{\DetailSomeInline}[1]{
    \ifthenelse{\value{DetailLevel}=\value{dtlSomeDetail} \or \value{DetailLevel}>\value{dtlSomeDetail}}
        {{\small #1}}{}
    }

\newcommand{\DetailFullInline}[1]{
    \ifthenelse{\value{DetailLevel}=\value{dtlFullDetails} \or \value{DetailLevel}>\value{dtlFullDetails}}
        {{\small #1}}{}
    }

\newcommand{\DetailSomeElseInline}[2]{
    \ifthenelse{\value{DetailLevel}=\value{dtlSomeDetail} \or \value{DetailLevel}>\value{dtlSomeDetail}}
        {{\small #1}}{#2}
    }

\newcommand{\DetailFullElseInline}[2]{
    \ifthenelse{\value{DetailLevel}=\value{dtlFullD	etails} \or \value{DetailLevel}>\value{dtlFullDetails}}
        {{\small #1}}{#2}
    }

\newcommand{\ExplainDetailLevel}{
    Detail level is
    \ifthenelse{\value{DetailLevel}=\value{dtlForSubmission}}
        {0: for submission}
        {\ifthenelse{\value{DetailLevel}=\value{dtlMarginComments}}
            {1: as for submission but with margin comments}
            {\ifthenelse{\value{DetailLevel}=\value{dtlSomeDetail}}
                {2: some proofs not intended for submission}
               {\ifthenelse{\value{DetailLevel}=\value{dtlFullDetails}}
                   {3: full details}
                   {invalid}
                }
            }
        }
    }

\newcounter{DoAdditionalConstraint} 

\newcommand{\AdditionalConstraint}[2]{
    \ifthenelse{\value{DoAdditionalConstraint}=1}
        {{\small #1}}{#2}
    }

\newtheorem{theorem}{Theorem}[section]
\newtheorem{prop}[theorem]{Proposition}
\newtheorem{lemma}[theorem]{Lemma}
\newtheorem{cor}[theorem]{Corollary}

\theoremstyle{definition}
\newtheorem{definition}[theorem]{Definition}

\newtheorem{remark}[theorem]{Remark}

\numberwithin{equation}{section}

\newcommand{\abs}[1]{\left\vert#1\right\vert}

\newcommand{\skipline}{\vspace{11pt}}

\newcommand{\BB}[1]{\ensuremath{\mathbb{#1}}}

\newcommand{\R}{\ensuremath{\BB{R}}}

\newcommand{\iny}{\ensuremath{\infty}}
\newcommand{\grad}{\ensuremath{\nabla}}
\ifthenelse{\value{HaveBBM}=1}{
    
    }{
    
    }
\DeclareMathOperator{\dv}{div} %
\DeclareMathOperator{\curl}{curl} %
\DeclareMathOperator{\supp}{supp} %
\newcommand{\prt}{\ensuremath{\partial}}
\newcommand{\brac}[1]{\ensuremath{\left[ #1 \right]}}
\newcommand{\pr}[1]{\ensuremath{\left( #1 \right) }}
\DeclarePairedDelimiter{\set}{\{}{\}}

\DeclarePairedDelimiter{\setMultiLine}{\{}{\}}
\newcommand{\norm}[1]{\ensuremath{\left\Vert #1 \right\Vert}}

\newcommand{\smallnorm}[1]{\ensuremath{\Vert #1 \Vert}}

\newcommand{\refA}[1]{Appendix~\ref{A:#1}}
\newcommand{\refS}[1]{Section~\ref{S:#1}}

\newcommand{\refT}[1]{Theorem~\ref{T:#1}}
\newcommand{\refTAnd}[2]{Theorems~\ref{T:#1} and \ref{T:#2}}

\newcommand{\refP}[1]{Proposition~\ref{P:#1}}
\newcommand{\refPAnd}[2]{Propositions~\ref{P:#1} and \ref{P:#2}}
\newcommand{\refPThrough}[2]{Propositions~\ref{P:#1} through \ref{P:#2}}
\newcommand{\refL}[1]{Lemma~\ref{L:#1}}

\newcommand{\refD}[1]{Definition~\ref{D:#1}}

\newcommand{\refE}[1]{(\ref{e:#1})}
\newcommand{\refEAnd}[2]{(\ref{e:#1}, \ref{e:#2})}

\newcommand{\refEThrough}[2]{(\ref{e:#1}-\ref{e:#2})}
\newcommand{\refR}[1]{Remark~\ref{R:#1}}

\newcommand{\eps}{\ensuremath{\varepsilon}}
\newcommand{\Cal}[1]{\ensuremath{\mathcal{#1}}}
\newcommand{\al}{\ensuremath{\alpha}}

\newcommand{\ol}{\overline}

\newcommand{\FTR}
    {\Cal{F}^{-1}}

\newcommand{\ThmProp}[1]{\textit{(#1)}}
\newcommand{\ThmPropThrough}[2]{\textit{(#1)}-\textit{(#2)}}

\newcommand{\LP}{\ensuremath{\Cal{P}}} %

\newcommand{\stardot}{\mathop{* \cdot}}

\begin{document}

\raggedbottom


\newcommand{\SideNote}[1]{
	    \ifthenelse{\value{DetailLevel}=\value{dtlMarginComments}
	    	\or \value{DetailLevel}>\value{dtlMarginComments}}
		{\pdfsyncstop\marginpar[\raggedleft\tiny #1]{\raggedright\tiny
		#1}\pdfsyncstart}
		{}
		}

\newcommand{\AuthorNote}[3]{\SideNote{\textcolor{blue}{(#1 #2):} #3}}

\newcommand{\JimNote}[2]{\AuthorNote{Jim}{#1}{#2}}

\newcommand{\MarginNote}[1]{\SideNote{#1}}

\newcommand{\ToDo}[1]{
	    \ifthenelse{\value{DetailLevel}=\value{dtlMarginComments}
	    	\or \value{DetailLevel}>\value{dtlMarginComments}}
		{\textbf{[#1]}}
		{}
		}


\newcommand{\Detail}[1]{
    \MarginNote{Detail}
    \skipline
    \hspace{+0.25in}\fbox{\parbox{4.25in}{\small #1}}
    \skipline
    }

\newcommand{\Holder}
    {H\"{o}lder\xspace}

\newcommand{\Holders}
    {H\"{o}lder's }

\newcommand{\wh}{\widehat}
 
\newcommand{\n}{\bm{n}}

\newcommand{\EqNum}[1]{\buildrel {#1}\over =}

\newcommand{\Ignore}[1] {}

\newcommand{\HelenasDualSpace}{(W^{1,1}_{\sigma})^*}

%
%

%
%

\author[Kelliher]{James P. Kelliher}
\address{Department of Mathematics, University of California, Riverside, 900 University Ave.,
Riverside, CA 92521}
\email{kelliher@math.ucr.edu}



\title
    [Bounded vorticity, bounded velocity 2D Euler]
    {A characterization at infinity of bounded vorticity, bounded velocity solutions to the 2D Euler equations}
     
\begin{abstract}
We characterize the possible behaviors at infinity of weak solutions to the 2D Euler equations in the full plane having bounded velocity and bounded vorticity. We show that any such solution can be put in the form obtained by Ph. Serfati in 1995 after a suitable change of reference frame. Our results build on those of a recent paper of the author's, joint with Ambrose, Lopes Filho, and Nussenzveig Lopes.
\end{abstract}

\date{(compiled on \today)}

\maketitle

\vspace{-0.3in}

\Ignore{ 
\begin{small}
    \begin{flushright}
        Compiled on \textit{\textbf{\today}}
    \end{flushright}
\end{small}
} 

\DetailMarginNote{
    \begin{small}
        \begin{flushright}
            Compiled on \textit{\textbf{\today}}

            \ExplainDetailLevel

        \end{flushright}
    \end{small}
    }
    
\newcommand{\BoxComment}[1]{
    \skipline
    \hspace{+0.25in}\fbox{\parbox{4.25in}{\small #1}}
    \skipline
    }
    
\vspace{-0.40in}

\renewcommand\contentsname{}   
{\small
\tableofcontents
}

\newpage

\Ignore{ 
\bigskip
\begin{center}
\textbf{Overview}
\phantomsection   
\addcontentsline{toc}{chapter}{Overview}
\end{center}
} 

%
%
\section{Introduction}\label{S:Introduction}

\noindent In classical form, the Euler equations (without forcing) can be expressed as
\begin{align}\label{e:EClassical}
    \begin{split}
        \left\{
        \begin{array}{rl}
            \prt_t u + u \cdot \grad u + \grad p &= 0, \\
            \dv u &= 0, \\
            u(0) &= u^0.
        \end{array}
        \right.
    \end{split}
\end{align}
Here, $u$ is a velocity field, $p$ is a scalar pressure field, and the initial velocity, $u^0$, is assumed to be divergence-free. We are concerned here exclusively with solutions in the full plane.

The nature of the solutions to these equations will depend strongly on the function spaces to which the initial data belongs. For functions spaces for which well-posedness results are known, nearly all studies
have assumed that the vorticity, $\omega = \curl u := \prt_1 u^2 - \prt_2 u^1$, decays at infinity rapidly enough that the velocity can be recovered from the vorticity via the Biot-Savart law,
\begin{align*}
    u = K * \omega,
\end{align*}
where $K$ is the Biot-Savart kernel (see \refE{K}). One commonly imposed condition that insures this is that $\omega \in L^{p_1} \cap L^{p_2}$ for some $p_1 < 2 < p_2$, in which case the velocity will also decay at infinity. (The Biot-Savart law can hold with some decay of the vorticity but without decay of the velocity at infinity, and solutions to the Euler equations can still be obtained: see \cite{Brunelli2010}.)

We will be concerned here with initial data for which the Biot-Savart law does not hold, treating the case where the vorticity and velocity are both bounded: what we call \textit{bounded} solutions. The construction of such solutions in the full plane was first decribed by Ph. Serfati in \cite{Serfati1995A}, proven in more detail in \cite{AKLL2014} (including the case of an exterior domain). An alternate construction, relying upon another Serfati paper, \cite{Serfati1995B}, was given by Taniuchi in \cite{Taniuchi2004}.

In each of \cite{Serfati1995A, Taniuchi2004, AKLL2014}, however, the behavior at infinity of a solution was assumed either implicitly or explicitly. Identical assumptions, on the velocity, are made in \cite{Serfati1995A, AKLL2014}, while \cite{Taniuchi2004} makes an assumption on the pressure. (We describe these assumptions in detail below.) These assumptions are a priori, in that they are used in the construction of the solutions. The first purpose of this work is to characterize a postierori all possible behaviors of bounded solutions at infinity, so as to avoid the need for such assumptions a priori.

The second purpose of this work is to show that, in fact, the bounded solutions constructed in \cite{Serfati1995A, AKLL2014} are identical to those constructed in \cite{Taniuchi2004}. This will require us to obtain the properties of the pressure for the solutions constructed in \cite{Serfati1995A, AKLL2014} and show that they match those of \cite{Taniuchi2004}.

\Ignore { 
\noindent A complete formulation of the solution to a PDE in an unbounded domain must include some condition on the behavior at infinity.
Such a ``boundary condition at infinity'' allows, under suitable additional assumptions, for both existence and uniqueness of solutions.


A condition at infinity is often chosen for mathematical expedience; in the extreme case, compact support of the solution
is assumed. This
focuses the attention on the local behavior of a solution, which is generally of central importance. It prevents a full understanding, however, of what it means to solve the PDE.

It would be preferable to first characterize what types of behavior at infinity are consistent with the PDE, just as is done with boundary conditions.
Only then
would a condition at infinity be selected to insure well-posedness. Ideally, such a condition would also be physically justified. Without characterizing all possible behaviors at infinity, however, one cannot assess the degree to which physical and mathematical necessity align.


In this paper, we perform such a characterization at infinity for one fairly involved example: weak solutions to the 2D Euler equations in the plane having bounded vorticity and bounded velocity (what we call \textit{bounded solutions}). We note that the mere membership of the solution in the required function space already poses a constraint at infinity; namely, that the solution remain bounded. But membership in $L^\iny$ is really a locally uniform constraint on the size of the solution and is, in any case, an insufficient condition at infinity, as we shall see (the full condition is given in \refE{CharacterizationR2}). 
 
In classical form, the Euler equations (without forcing) can be expressed as
\begin{align}\label{e:EClassical}
    \begin{split}
        \left\{
        \begin{array}{rl}
            \prt_t u + u \cdot \grad u + \grad p &= 0, \\
            \dv u &= 0, \\
            u(0) &= u^0.
        \end{array}
        \right.
    \end{split}
\end{align}
Here, $u$ is a velocity field, $p$ is a scalar pressure field, and the initial velocity, $u^0$, is assumed to be divergence-free.
For our definition of a bounded weak Euler solution we make the near minimal assumptions that it satisfy the equations (in vorticity form) weakly, the velocity be continuous in time, and the vorticity be transported by the flow (\refD{ESol}).
} 


To understand what types of behavior at infinity we might expect, consider
the following two classical solutions $(u_1, p_1)$ and $(u_2, p_2)$ to \refE{EClassical}:
\begin{align}\label{e:JunKatoExamples}
    \begin{array}{ll}
        u_1(t, x) = u^0 + U_\iny(t), &p_1(t, x) = - U_\iny'(t) \cdot x, \\   
        u_2(t, x) = u^0, &p_2(t, x) = 0.
    \end{array}   
\end{align}
Here, $U_\iny$ is any differentiable vector-valued function of time for which $U_\iny(0) = 0$.
Both are easily verified to be solutions to the Euler (and, for that matter, Navier-Stokes) equations as in \refE{EClassical} with the same initial velocity, $u^0$. In \cite{GKKM2001, JunKato2003}, the authors use these examples to make the point that to insure solutions are unique, some condition on the pressure must be imposed for solutions to the Navier-Stokes equations in the plane.
 
Here, we draw a different lesson from this example, one that leads to a characterization of all possible bounded solutions to the Euler equations. We prove that the solution's behavior at infinity is of necessity very much like that of $(u_1, p_1)$.

Specifically, for solutions in the full plane, we show that there exists some continuous vector-valued function of time, $U_\iny$, with $U_\iny(0) = 0$, for which 
\begin{align}\label{e:CharacterizationR2}
	\begin{split}
		u(t, x) - u(0, x) &= U_\iny(t)
		    + \lim_{R \to \iny} (a_R K) * (\omega(t) - \omega(0))(x), \\
		\grad p(t, x) &= - U_\iny'(t) + O(1), \\
		p(t, x) &= - U_\iny'(t) \cdot x + O(\log \abs{x}),
	\end{split}
\end{align}
the explicit expression for the $O(1)$ (in $\abs{x}$) function being given in \refE{gradpR2}.
In \refE{CharacterizationR2}, $\omega(t) = \prt_1 u^2(t) - \prt_2 u^1(t)$ is the vorticity (scalar curl) of the velocity field $u(t)$, $K$ is the Biot-Savart kernel (see \refE{K}), and $a_R$ is any cutoff function with support increasing to infinity with $R$, as in \refD{RadialCutoff}. 
The time derivative on $U_\iny$ in $\refE{CharacterizationR2}_{2, 3}$ is a distributional derivative.

To explain what $\refE{CharacterizationR2}_1$ means, we need one basic fact concerning the Biot-Savart law: If $\omega \in L^1 \cap L^\iny(\R^2)$ then
$
    u = K * \omega
$
is the unique, divergence-free vector field vanishing at infinity whose vorticity is $\omega$.

The condition that $\omega$ be in $L^1 \cap L^\iny$ can be weakened, but some decay at infinity is required for the Biot-Savart law to hold. Hence, we have no hope of applying the Biot-Savart law for our solutions, as we wish to assume no decay of vorticity. But we will discover a replacement for the Biot-Savart law that will work, and name it the \textit{renormalized Biot-Savart law}, defined as follows:
\begin{quote}
We say that the renormalized Biot-Savart law holds for a vector field, $v$, if there exists a constant vector field, $H$, such that
\begin{align}\label{e:RenBSLaw}
    v = H + \lim_{R \to \iny} (a_R K) * \omega(v)
\end{align}
pointwise in $\R^2$, where $\omega(v) := \prt_1 v^2 - \prt_2 v^1$.\end{quote}

\medskip

When $\omega(v)$ has sufficient decay at infinity, \refE{RenBSLaw} holds without the need for a cutoff function: we simply obtain $v = H + K * \omega$, with $H$ being the value of $v$ at infinity.

The relation in $\refE{CharacterizationR2}_1$, then, says that the renormalized Biot-Savart law holds for the vector field $u(t) - u(0)$ at any time, $t$, with $H = U_\iny(t)$.

\Ignore{ 
Now, $u(t) - u(0)$ has no clearly defined value at infinity, for
\begin{align*}
    \lim_{\abs{x} \to \iny}
        \lim_{R \to \iny} (a_R K) * (\omega(t) - \omega(0))(x)
\end{align*}
does not, in general, exist. But in the sequence of approximate solutions, $(u_n)$, used to construct $u$ (see \refS{ExistenceR2}), we do have
\begin{align*}
    \lim_{n \to \iny}
        \lim_{\abs{x} \to \iny} (u_n(t, x) - u_n(0, x)) = U_\iny(t).
\end{align*}
Because of this, we will call $U_\iny(t)$ the \textit{weak value at infinity} for the velocity field, $u(t) - u(0)$, or say that $U_\iny(t)$ is the \textit{weak velocity at infinity relative to $u(0)$}.
} 

The velocity field, $U_\iny$, can be eliminated in \refE{CharacterizationR2} (or in $\refE{JunKatoExamples}_1$) by changing to an accelerated frame of reference by the transformation,
\begin{align}\label{e:COV}
    \begin{split}
        \ol{x} &= \ol{x}(t, x) = x + \int_0^t U_\iny(s) \, ds, \\
        \ol{u}(t, x) &= u(t, \ol{x}) - U_\iny(t),
            \quad
        \ol{p}(t, x) = p(t, \ol{x}) + U_\iny'(t) \cdot x.
    \end{split}
\end{align}
(See the first part of \refL{COVEquiv}.)
Note that this is a Galilean transformation when $U_\iny$ is constant in time.
Setting $\ol{\omega} = \omega(\ol{u})$, the chain rule gives $\ol{\omega}(t, x) = \omega(t, \ol{x})$, and it follows that
\begin{align}\label{e:CharacterizationR2NoUInf}
	\begin{split}
		\ol{u}(t, x) - \ol{u}(0, x)
		    &= \lim_{R \to \iny} (a_R K)
		        * (\ol{\omega}(t) - \ol{\omega}(0))(x), \\
		\grad \ol{p}(t, x) &= O(1),
		    \quad
		\ol{p}(t, x) = O(\log \abs{x}),
	\end{split}
\end{align}
and $(\ol{u}, \ol{p})$ satisfy the Euler equations in the sense of distributions. Physically, this reflects the fact that a change of frame by translation, even an accelerated translation, introduces a force that is a gradient, and so is absorbable into the pressure gradient.
\DetailSome{
    We have,
    \begin{align*}
        \prt_t \ol{u}(t, x)
            &= \prt_t u(t, \ol{x}) + Du(t, \ol{x}) D \ol{x} - U_\iny'(t) \\
            &= \prt_t u(t, \ol{x}) + U_\iny(t) \cdot \grad u(t, \ol{x})
                - U_\iny'(t),
    \end{align*}
    \begin{align*}
        \grad \ol{u}(t, x)
            &= \grad \ol{x} \grad u(t, \ol{x})
            = \grad u(t, \ol{x}),
    \end{align*}
    and
    \begin{align*}
        \grad \ol{p}(t, x)
            &= \grad p(t, x) + U_\iny'(t).
    \end{align*}
    Hence,
    \begin{align*}
        \prt_t \ol{u}(&t, x) + \ol{u}(t, x) \cdot \grad \ol{u}(t, x)
                + \grad \ol{p}(t, x) \\
            &= \prt_t u(t, \ol{x}) + U_\iny(t) \cdot \grad u(t, \ol{x})
                    - U_\iny'(t) \\
            &\qquad\qquad
                + \pr{u(t, \ol{x}) - U_\iny(t)} \cdot \grad u(t, \ol{x})
                + \grad p(t, x) + U_\iny'(t) \\
            &= \prt_t u(t, x) + u(t, x) \cdot \grad u(t, x) + \grad p(t, x)
            = 0.
    \end{align*}
}

Alternately, we can view solutions for which $U_\iny$ is not identically zero to be in an accelerated frame: we then move to an inertial frame, in which $U_\iny \equiv 0$, by the transformation above. Such solutions in an inertial frame are identical to those constructed by Serfati in \cite{Serfati1995A} (a more complete derivation appears in \cite{AKLL2014}). Observe as well that the two solutions in \refE{JunKatoExamples} are the same solution after the transformation in \refE{COV}.

That $U_\iny$ can be eliminated by changing frames in this way is an a posteriori conclusion reached only after establishing the existence of such a vector field for which \refE{CharacterizationR2} holds. Since we cannot transform $U_\iny$ away until we obtain it, obtaining it is unavoidable. Moreover, it is in demonstrating that \refE{CharacterizationR2} must hold for some $U_\iny$ that we say we \textit{characterize} solutions to the Euler equations at infinity.

To cast a different light on our characterization of solutions, consider the special case of sufficiently decaying (say, compactly supported) initial vorticity in the full plane. Then the classical Biot-Savart law applies, and $\refE{CharacterizationR2}_1$ reduces to
$
    u(t) = U_\iny(t) + K * \omega(t).
$
This gives the usual characterization of solutions to the 2D Euler equations for decaying vorticity whose velocity at infinity is $U_\iny$ (often chosen to be zero). Actually, this is not normally viewed as a characterization of the solution, but rather as a way of recovering the velocity from the vorticity, and so obtaining a formulation of the Euler equations solely in terms of the vorticity. This same point of view applies for our non-decaying bounded solutions as well (see \refR{CharacterizingSolutions}.)

Key to our characterization of the velocity field for a solution, $u$, to the 2D Euler equations in the full plane is the observation that any bounded velocity field, $v$, having bounded vorticity satisfies the renormalized Biot-Savart law \refE{RenBSLaw} for a subsequence, $(R_k)$. Applying this to $v = u(t) - u(0)$ and using properties of the Euler equations allows us to show that $\refE{CharacterizationR2}_1$ holds.

Having obtained the characterizations in $\refE{CharacterizationR2}_1$, the task of establishing existence and uniqueness immediately arises. We will find this task easy, however, because existence and uniqueness in the special case of $U_\iny \equiv 0$ was already proved in \cite{AKLL2014} (for both the full plane and the exterior of a single obstacle). The transformation in \refE{COV} makes this especially simple.

The characterizations in $\refE{CharacterizationR2}_1$ along with existence and uniqueness give a fairly complete picture of the velocity for bounded solutions to the Euler equations. For the pressure, we take a much different approach, for we will not find it possible to directly characterize the pressure as we did the velocity. Limiting us in this regard is the lack of decay at infinity of the velocity field (from which the pressure is ultimately derived).

Instead, we will show that the solutions we construct in our proof of existence also satisfy $\refE{CharacterizationR2}_{2, 3}$. We do this using the sequence of smooth approximate solutions, which decay sufficiently rapidly at infinity, and taking a limit. Because we have uniqueness of solutions using only $\refE{CharacterizationR2}_1$, it follows that $\refE{CharacterizationR2}_{2, 3}$ hold for all bounded solutions.

\bigskip

We say now a few words about works in the literature pertaining to bounded solutions to the 2D Euler equations and how they relate to this work.

Our proof of the existence and uniqueness of solutions in \refS{ExistenceR2} is a modest extension of the proof in \cite{AKLL2014}, which in turn builds on the approach in \cite{Serfati1995A}, where the existence and uniqueness of such solutions was first proved by Serfati in the full plane. Serfati's full-plane existence result was extended by Taniuchi in \cite{Taniuchi2004} to allow slightly unbounded vorticity (a localized version of the velocity fields treated by Yudovich in \cite{Y1995}), while Taniuchi with Tashiro and Yoneda in \cite{TaniuchiEtAl2010} established uniqueness (and more). In \cite{AKLL2014}, Serfati's result was obtained both for the full plane and for the exterior to a single obstacle.

In each of these papers, the solutions that were constructed had a special property that was used as a selection criterion to guarantee uniqueness. In \cite{Taniuchi2004, TaniuchiEtAl2010}, that property was that the pressure belong to $BMO$ and was given by a Riesz transform in the classical way. (This implies at most logarithmic growth of the pressure at infinity, as we show.) In \cite{AKLL2014}, an identity (\refE{SerfatiId}, below, with $U_\iny \equiv 0$) that we show is equivalent to $\refE{CharacterizationR2}_1$ was used. This identity, called the \textit{Serfati identity} here and in \cite{AKLL2014}, was implicitly used, though never explicitly stated, by Serfati both in the construction of a solution (in the full plane) and to establish uniqueness\footnote{Serfati seems to state that the sublinear growth of the pressure is his uniqueness criterion, but uses an estimate derived from the Serfati identity in his proof of uniqueness.}; the same is done, explicitly, in \cite{AKLL2014}. The desire to remove the need for this identity was one motivation for this paper.

Our characterization in \refE{CharacterizationR2} of solutions helps to clarify the roles played by these selection criteria in the full plane. Taniuchi, Tashiro, and Yoneda use, in effect, the selection criterion, $U_\iny' \equiv 0$, whereas, in \cite{AKLL2014}, the criterion is $U_\iny \equiv 0$. From our characterization in \refE{CharacterizationR2}, these are, in fact, equivalent, since $U_\iny(0) = 0$.

The proof of uniqueness in \cite{TaniuchiEtAl2010} is for the bounded solutions constructed by Taniuchi in \cite{Taniuchi2004}, which are not known to coincide with the solutions constructed in \cite{Serfati1995A, AKLL2014}. Also motivating this paper was the desire to show that the bounded solutions constructed in \cite{Serfati1995A, Taniuchi2004, AKLL2014} do, in fact, coincide. Accomplishing this requires us to obtain the pressure corresponding to a solution constructed in \cite{AKLL2014} and show that it has the same properties as those established in \cite{Taniuchi2004}.

In \refS{TTYWork}, we discuss further some issues related to \cite{Taniuchi2004, TaniuchiEtAl2010} that are best understood after the proof of our results have been presented. Further, in \refS{JunKato}, we discuss the relation of our approach to obtaining properties of the pressure with the approach taken by Jun Kato in \cite{JunKato2003} for solutions to the Navier-Stokes equations in the plane for bounded initial velocity.
 
The vanishing viscosity limit of the Navier-Stokes equations to the Euler equations has been studied for bounded solutions in \cite{Cozzi2009, Cozzi2010, Cozzi2013}.

\bigskip

This paper is organized as follows:

In \refS{Results} we define our bounded solutions to the 2D Euler equations and state our main results. We summarize some background facts and definitions in \refS{Background} that we will use throughout the paper.

In \refS{CharacterizationSectionR2}, we characterize bounded solutions for the full plane, giving the proof of existence and uniqueness in \refS{ExistenceR2}. In \refS{PressureR2}, we establish the properties of the pressure for the full plane. The formula for the pressure gradient in the full plane is the same as that in \cite{Serfati1995B}, and is based on the Green's function for the Laplacian. The most delicate estimates, those characterizing the behavior of the pressure itself at infinity, we obtain using a Riesz transform. These estimates are presented in \refS{PoissonFullPlane}.

In \refS{Afterword}, we make a few final comments concerning the nature of the weak solutions we have defined.
\refA{SomeLemmas} contains several lemmas we use elsewhere in this paper.

%
%
\section{Statement of results}\label{S:Results}

\noindent
Before stating our results, we must make several definitions.

For a velocity field, $u$, the vorticity, $\omega(u) = \curl(u) := \prt_1 u^2 - \prt_2 u^1$.

Let $G(x, y) = (2 \pi)^{-1} \log \abs{x - y}$, the Green's function for the Laplacian in the full plane. Then the Biot-Savart kernel in the full plane is given by
\begin{align}\label{e:K}
	K(x)
	    = \grad^\perp G(x)
	    = \frac{1}{2 \pi} \frac{x^\perp}{\abs{x}^2},
\end{align}
where $\grad^\perp := (-\prt_2, \prt_1)$ and $x^\perp := (-x_2, x_1)$. When $\omega$ is a compactly supported, bounded scalar field, we define
\begin{align}\label{e:BSLawFullPlane}
    K[\omega]
        = K * \omega.
\end{align}
Then $K[\omega]$ is the unique, divergence-free vector field vanishing at infinity whose vorticity is $\omega$.

\begin{definition}\label{D:SerfatiVelocity}
	We say that a divergence-free vector field, $u \in L^\iny(\R^2)$, with vorticity,
	$\omega(u) \in L^\iny(\R^2)$
	is a \textit{Serfati velocity}.
	We call the space of all such vector fields, $S = S(\R^2)$, with the norm,
	\begin{align*}
		\norm{u}_S
			= \norm{u}_{L^\iny} + \norm{\omega(u)}_{L^\iny}.
	\end{align*}
\end{definition}

\begin{definition}\label{D:SerfatiConvergence}	
	We say that a sequence, $(u_n)$, in $L^\iny(0, T; S)$
	\textit{converges locally in $S$} if for any compact subset,
	$L$, of $\R^2$,
	\begin{align*}
	    \norm{u_n - u}_{L^\iny([0, T] \times L)}
	        + \norm{\omega(u) - \omega(u_n)}_{L^\iny([0, T] \times L)}
	        \to 0. 
    \end{align*}	
\end{definition}

We will use the following definition for solutions in the full plane:

\begin{definition}\label{D:ESol}
	Fix $T > 0$. 
	We say that a velocity field, $u$, lying in $L^\iny(0, T; S) \cap
	C([0, T] \times \R^2)$
	having vorticity, $\omega = \omega(u)$,
	is a \textit{bounded solution} to the Euler equations
	without forcing if, on the interval, $[0, T]$,
	$
		\prt_t \omega + u \cdot \grad \omega
			= 0
	$
	as distributions on $(0, T) \times \R^2$ and the vorticity is
	transported by the flow map. 
\end{definition}

\begin{remark}\label{R:FlowMap}
Because the velocity, $u$, of \refD{ESol} lies in $L^\iny(0, T; S) \cap C([0, T] \times \R^2)$, it follows from \refL{Morrey} that $u$ has a spatial log-Lipschitz modulus of continuity (MOC) with a uniform bound over $[0, T]$ and thus that it has a unique classical flow map; hence, the existence of a flow map need not be made a requirement in \refD{ESol}.
\end{remark}

\begin{remark}\label{R:CharacterizingSolutions}
The vorticity equation, $\prt_t \omega + u \cdot \grad \omega = 0$, in \refD{ESol} is not a vorticity formulation, since we do not specify how the velocity field is recovered from the vorticity. Indeed, the key fact we show in this paper is that the membership of $u(t)$ in $S$ forces the recovery of the velocity from the vorticity to take place in the specific manner given by $\refE{CharacterizationR2}_1$  (more precisely stated in \refT{CharacterizationR2}). The only freedom is the choice of $U_\iny$. (We can use this observation to define a vorticity formulation, as we explain in \refS{VorticityFormulation}.)
\end{remark}

\begin{definition}\label{D:RadialCutoff}
    Let $a$ be a radially symmetric, smooth, compactly supported
    function with $a = 1$ in a neighborhood of the origin. We will
    refer to such a function simply as a \textbf{radial cutoff
    function}. For any $R > 0$ we define
    \begin{align*}
        a_R(\cdot) = a(\cdot/R).
    \end{align*}
\end{definition}

\begin{definition}\label{D:stardot}
	For $v$, $w$ vector fields, we define
	$
		v \stardot w
			= v^i * w^i.
	$
	For $A$, $B$ matrix-valued functions on $\R^2$, we define
	$
		A \stardot B
			= A^{ij} * B^{ij}.
	$
	Here, and throughout this paper, we use the convention that repeated
	indices are summed over.
\end{definition}

\bigskip

Our main results are \refTAnd{CharacterizationR2}{ExistenceR2}.

\begin{theorem}[Characterization of solutions]\label{T:CharacterizationR2}
	Suppose that $u$ is a solution to the Euler equations as in
	\refD{ESol} in the full plane
	with initial velocity, $u(t = 0) = u^0 \in S$, and initial vorticity,
	$\omega^0 = \omega(u^0)$.
	There exists $U_\iny \in (C[0, T])^2$ with $U_\iny(0) = 0$,
	such that each of the following holds:
	
	\noindent
	(i) Serfati identity: for $j = 1$, $2$,
	\begin{align}\label{e:SerfatiId}
		\begin{split}
			u^j(t&) - (u^0)^j
				= U^j_\iny(t)
					+ (a K^j) *(\omega(t) - \omega^0) \\
				&\qquad- \int_0^t \pr{\grad \grad^\perp \brac{(1 - a) K^j}}
				\stardot (u \otimes u)(s)
							\, ds.
		\end{split}
	\end{align}
	
	\noindent
	(ii) Renormalized Biot-Savart law:
	\begin{align}\label{e:RenormalizedBS}
		\begin{split}
			u(t) - u^0
				= U_\iny(t)
					+ 
					\lim_{R \to \iny} (a_R K) *(\omega(t) - \omega^0)
		\end{split}
	\end{align}
	on $[0, T] \times \R^2$ for all radial cutoff
	functions, $a$, as in \refD{RadialCutoff}.
	The convergence in \refE{RenormalizedBS}
	is locally uniform in $S$ as in \refD{SerfatiConvergence}.
	
	\noindent
	(iii) There exists
	a pressure field, $p \in \Cal{D}'((0, T) \times \R^2)$
	with $\grad p + U_\iny'$ lying
	in $L^\iny([0, T] \times \R^2)$, such that
    \begin{align}\label{e:VelocityEquationR2}
        \prt_t u + u \cdot \grad u + \grad p = 0  
    \end{align}
	as distributions on $(0, T) \times \R^2$. Here,
	$\prt_t u - U_\iny'\in L^\iny(0, T; L^p_{loc}(\R^2))$
	for all $p$ in $[1, \iny)$.
	(Note that $U_\iny' \in (\Cal{D}'((0, T)))^2$.)
	
	\noindent
	(iv)
	For any radial cutoff function, $a$, as in \refD{ESol},
	\begin{align}\label{e:gradpR2}
		\begin{split}
		\grad p(t, x)
			&= -U_\iny'(t)
				+ \int_{\R^2} a(x - y) K^\perp(x - y)
				    \dv \dv (u \otimes u)(t, y) \, dy \\
			&\quad
	 			+ \int_{\R^2} (u \otimes u)(t, y) \cdot
					\grad_y \grad_y \brac{(1 - a(x - y))
					K^\perp(x - y)}
					\, dy.
		\end{split}
	\end{align}
	Also,
	$\norm{\grad p(t) + U_\iny'(t)}_{L^\iny}
	\le C \smallnorm{u^0}_S^2$.
	
	\noindent
	(v) Pressure growth at infinity: The pressure, $p$,
	can be chosen so that
	\begin{align}\label{e:pRieszRel}
        p = -U_\iny' \cdot x
			        - R (u \otimes u),
    \end{align}
    where $R = \Delta^{-1} \dv \dv$ is a Riesz transform on
    $2 \times 2$ matrix-valued functions on $\R^2$.
    Moreover,
    \begin{align}\label{e:pBMO}
        p(t, x) &+ U_\iny'(t) \cdot x \in L^\iny([0, T]; BMO)   
    \end{align}
    with
	\begin{align}\label{e:pBoundR2}
	    	p(t, x)
	    	    &= -U_\iny'(t) \cdot x + O(\log \abs{x}),
	\end{align}
\end{theorem}

\begin{theorem}\label{T:ExistenceR2}
    Assume that $u^0 \in S$, let $T > 0$ be arbitrary, and fix
    $U_\iny \in (C[0, T])^2$ with $U_\iny(0) = 0$. There exists a
    bounded solution, $u$, to the Euler equations as in \refD{ESol},
    and this solution satisfies (i)-(v) of
    \refT{CharacterizationR2}.
    This solution is unique among all solutions
    with $u(0) = u^0$ that satisfy any one of the following
    uniqueness criteria:
    \begin{itemize}
        \item[(a)]    
            (i) of \refT{CharacterizationR2} holds;

        \item[(b)]
            (ii) of \refT{CharacterizationR2} holds;

        \item[(c)]
            there exists a pressure satisfying
            \refEAnd{VelocityEquationR2}{pRieszRel} for which
            \refE{pBMO} holds;
            
        \item[(d)]
            there exists a pressure satisfying
            \refEAnd{VelocityEquationR2}{pRieszRel} for which 
            $\grad p + U_\iny' \in L^\iny([0, T] \times \R^2)$
            and \refE{pBoundR2} holds.
    \end{itemize}
\end{theorem}

\begin{remark}
	Radial symmetry of the cutoff function, $a$, simplifies some of our proofs,
	so we adopt it, but it is not a necessary assumption.
\end{remark}

\refT{CharacterizationR2} shows that if one has a bounded solution to the Euler equations then there must be a $U_\iny$ for which the solution has the stated properties. \refT{ExistenceR2} is a kind of converse, which says that if one has a $U_\iny$ there does, in fact, exist a bounded solution to the Euler equations that satisfies one of the properties stated in \refT{CharacterizationR2}. By the uniqueness in \refT{ExistenceR2} it then follows that the solutions whose existence is ensured by that theorem satisfies all of the properties given in \refT{CharacterizationR2}.

We begin the proof of \refT{CharacterizationR2} in \refS{CharacterizationSectionR2} by establishing properties \ThmProp{i} and \ThmProp{ii}, thereby characterizing the velocity for bounded solutions in the full plane. \refT{ExistenceR2}, giving the existence of solutions along with uniqueness of such solutions that satisfy \refE{SerfatiId}, follows easily from the construction of Serfati solutions in \cite{AKLL2014} and the transformation in \refE{COV}: this is explained in detail in \refS{ExistenceR2}. It follows from this uniqueness, then, that any further properties we can establish for the Serfati solutions constructed in \cite{AKLL2014}, modified by \refE{COV}, must hold for our bounded solutions. In \refS{PressureR2} we establish some such properties; namely, those of the pressure appearing in \ThmProp{iii}-\ThmProp{v} of \refT{CharacterizationR2}.

The formula for the pressure gradient in the full plane is the same as that in \cite{Serfati1995B}, and is based on the Green's function for the Laplacian. The most delicate estimates, those characterizing the behavior of the pressure itself at infinity, we obtain using Riesz transforms in the full plane. These estimates appear in \refS{PoissonFullPlane}.

%
%
\section{Background Material}\label{S:Background}

\noindent In this section we present definitions and bounds that we will need in the remainder of this paper.

We have the following estimates on $K$ of \refE{K}:

\begin{prop}\label{P:KStarBounds}
	We have,
	\begin{align}\label{e:JJBound}
		\abs{K(x - y)}
			&\le \frac{C}{\abs{x - y}}.
	\end{align}
	Let $a$ be a radial cutoff function.
	There exists $C>0$ such that for all $\eps > 0$,
	\begin{align}
		\norm{\grad_y a_\eps(x - y) \otimes
		    \grad_y K^i(x - y)}_{L^1_y(\R^2)}
			&\le C \eps^{-1},
								\label{e:D1JJBound} \\
		\norm{\grad_y \grad_y \brac{(1 - a_\eps(x - y))
		    K(x - y)}}_{L^1_y(\R^2)}
			&\le C \eps^{-1}.
								\label{e:D2KKBound}
	\end{align}
	Let $U \subseteq \R^2$ have measure $2 \pi R^2$ for some
	$R < \iny$.
	Then for any $p$ in $[1, 2)$,
	\begin{align}\label{e:RearrangementBound}
		\smallnorm{K(x - \cdot)}_{L^p(U)}^p
			&\le \frac{R^{2 - p}}{2 - p}.
	\end{align}
\end{prop}
\begin{proof}
    The bound in \refE{JJBound} is immediate from \refE{K}.
    For the bounds in \refEThrough{D1JJBound}{RearrangementBound}
    see \cite{AKLL2014}.
\end{proof}

\begin{definition}\label{D:MOC}
    	A nondecreasing continuous function,
    	$\mu \colon [0, \iny) \to [0, \iny)$,
    	is a modulus of continuity (MOC)
    	if $\mu(0) = 0$ and $\mu > 0$ on $(0, \iny)$. 
\end{definition}

\refD{Cmu} is a generalization of \Holder-continuous functions.

\begin{definition}\label{D:Cmu}
	Let $\mu$ be a MOC. Define
	\begin{align*}
	    C_\mu
	        = C_\mu(\R^2)
			= \setMultiLine{f \in C(\R^2) \colon &\exists \, c_0 > 0
			    \textit{ s.t. } \forall \, x, y \in \R^2, \\
			    &\abs{f(x) - f(y)} \le c_0 \mu(\abs{x - y})}
	\end{align*}
	with
	\begin{align*}
		\norm{f}_{C_\mu}
			= \norm{f}_{L^\iny}
				+ \norm{f}_{\dot{C}_\mu},
	\end{align*}
	where
	\begin{align*}
		\norm{f}_{\dot{C}_\mu}
			= \sup_{x \ne y} \frac{\abs{f(x) - f(y)}}{\mu(\abs{x - y})}.
	\end{align*}
\end{definition}

We define Log-Lipschitz functions explicitly by using the MOC,
\begin{align}\label{e:muLL}
	\mu_{LL}(r)
		= \left\{
			\begin{array}{rl}
				-r \log r, & \text{if } r \le e^{-1}, \\
				e^{-1}, & \text{if } r > e^{-1},
			\end{array}
		\right.
\end{align}
setting $LL = C_{\mu_{LL}}$ and $\dot{LL} = \dot{C}_{\mu_{LL}}$.

\begin{definition}\label{D:Dini}
Given a MOC, $\mu$, we define,
\begin{align*}
    S_\mu(x)
        = \int_0^x  \frac{\mu(r)}{r} \, dr.
\end{align*}
We say that $\mu$ is \textit{Dini} if $S_\mu$ is finite for some (and hence all) $x > 0$. (Note that when $\mu$ is Dini, $S_\mu$ is itself a MOC.) A function is \textit{Dini-continuous} if it has a Dini MOC.
\end{definition}


%
%
\section{Characterization of velocity in the full plane}\label{S:CharacterizationSectionR2}

\noindent In this section we characterize the velocity, $u$, for solutions to the 2D Euler equations in the full plane, proving $\refE{CharacterizationR2}_1$, stated more precisely in \textit{(i), (ii)} of \refT{CharacterizationR2}. In outline, our proof proceeds as follows:

\begin{enumerate}

\item
In \refS{SerfatiIdR2} we show that if the Serfati identity, which we can write more concisely as
\begin{align}\label{e:SerfatiIdConcise}
    \begin{split}
        u(t&) - u^0
            = U_\iny(t) + (a K) *(\omega(t) - \omega^0) \\
        &\qquad- \int_0^t \pr{\grad \grad^\perp \brac{(1 - a) K}}
        \stardot (u \otimes u)(s) \, ds,
    \end{split}
\end{align}
holds then the renormalized Biot-Savart law for $u(t) - u(0)$,
\begin{align}\label{e:VelocityExpression}
		\displaystyle u(t) - u(0) = U_\iny(t)
		   + \lim_{R \to \iny} (a_R K) * (\omega(t) - \omega(0)),
\end{align}
holds without taking a subsequence.

\item
We also show in \refS{SerfatiIdR2} that, conversely, if \refE{VelocityExpression} holds for a subsequence then \refE{SerfatiIdConcise} holds. The subsequence can vary with time.

\item
We prove in \refS{BSLawR2} that for all $v \in S$ the renormalized Biot-Savart law, \refE{RenBSLaw}, holds for a subsequence; that is, we have $v = H + \lim_{k \to \iny} (a_{R_k} K) * \omega(v)$ for some subsequence, $(R_k)$, and constant vector field, $H$.
		
\item Let $v = u(t) - u^0$. Then \refE{VelocityExpression} holds for some subsequence possibly varying over time and some $U_\iny \colon [0, T] \to \R^2$ by step 3, so \refE{SerfatiIdConcise} holds by step 2, so \refE{VelocityExpression} holds for the full sequence by Step 1.

\item
Finally, since \refE{SerfatiIdConcise} holds and $u \in C([0, T]; L^\iny)$, we have $U_\iny \in C([0, T])^2$. We make this argument in \refS{CharacterizationR2}.
 
\end{enumerate}

%
%
\subsection{The Serfati identity in the full plane}\label{S:SerfatiIdR2}

\noindent In this subsection we prove \refP{RenormalizationLemma}, giving the equivalence between the renormalized Biot-Savart law and the Serfati identity. Formally, this equivalence follows from several integrations  by parts, but we must take some care to do these integrations in the face of the fairly minimal time regularity of the vorticity for our weak solutions. (The convolutions in space will all be of a compactly supported distribution with a tempered distribution, and so represent no difficulties.)

\begin{prop}\label{P:RenormalizationLemma}
	Suppose that $u$ is a solution to the Euler equations in the full plane
	as in \refD{ESol}.  Then if $u$ satisfies
	\refE{SerfatiId} for some $U_\iny$
	then \refE{RenormalizedBS} holds, the convergence being uniform on compact subsets
	of $[0, T] \times \R^2$.
	Conversely, if \refE{RenormalizedBS} holds for a subsequence
	for some $U_\iny$,
	the convergence being pointwise for any fixed $t \in [0, T]$,
	then $u$ satisfies
	\refE{SerfatiId}.
	The subsequence is allowed to vary with $t \in [0, T]$.
\end{prop}
\begin{proof}
	Assume that \refE{SerfatiId} holds.
	Because the vorticity
	is transported by the flow map and the velocity is continuous in time and space,
	both integrals in \refE{SerfatiId} are continuous as functions of $t$ and $x$. Therefore,
	it must be that $U_\iny \in C([0, T])$.
	
	By \refP{SerfatiForAlla}, \refE{SerfatiId} holds for $a_R$ in place
	of $a$ for all $R > 0$. Taking the limit as $R \to \iny$ and 
	applying \refE{D2KKBound} gives \refE{RenormalizedBS},
	the convergence being uniform on compact subsets
	of $[0, T] \times \R^2$.
	
	Now assume that \refE{RenormalizedBS} holds for a subsequence,
	$(R_k)$, with
	the convergence being pointwise for any fixed $t \in [0, T]$.
	Because $t$ is fixed in the argument that follows, it does not matter
	whether the subsequence varies with time.
	Fixing $x$ in $\R^2$ and letting
	$h(y) = (a_{R_k} - a)(x - y) K^j(x - y)$, $j = 1$ or $2$,
	\refL{gammaLemma} gives
	\begin{align}\label{e:renormEq}
		\begin{split}
			((a_{R_k} - a) K^j&) * (\omega(t) - \omega^0) \\
				&= \int_0^t \grad \grad^\perp \brac{(a_{R_k} - a) K^j} \stardot
					(u \otimes u)(s) \, ds.
		\end{split}
	\end{align}
	Because of \refE{RenormalizedBS}, as $k \to \iny$,
	the left hand side of \refE{renormEq} converges to
	\begin{align*}
		u^j(t, x) - (u^0)^j(x) - U_\iny(t)
			- (a K^j&) * (\omega(t) - \omega^0).
	\end{align*}
	
	The right-hand side of \refE{renormEq} can be written,
	\begin{align*}
		\int_0^t \grad &\grad^\perp \brac{(1 - a) K^j} \stardot
					(u \otimes u)(s) \, ds \\
			&-
		\int_0^t \grad \grad^\perp \brac{(1 - a_{R_k}) K^j} \stardot
					(u \otimes u)(s) \, ds.
	\end{align*}
	Applying \refE{D2KKBound} with Young's convolution inequality
	to the second term above we see that it vanishes as $R_{k} \to \iny$
	(here, we need only that $u \in L^\iny([0, T] \times \R^2$).
	Taking the limit as $k \to \iny$, then, it follows that \refE{SerfatiId} holds and hence also,
	as observed above, $U_\iny \in C([0, T])$.
\end{proof}

\begin{remark}\label{R:RenormalizationLemma}
	It follows from \refP{RenormalizationLemma} that if \refE{RenormalizedBS} holds for a subsequence,
	the convergence being pointwise for any fixed $t \in [0, T]$, then the convergence actually holds
	for the full sequence and is
	uniform on compact subsets of $[0, T] \times \R^2$.
\end{remark}

The key to the proof of \refP{RenormalizationLemma} was showing that if the Serfati identity holds for one cutoff function it holds for all cutoff functions: this is the purpose of \refP{SerfatiForAlla}, which rests on \refL{gammaLemma}, a technical lemma that handles integrating by parts in the face of the low time regularity of bounded solutions.

\begin{prop}\label{P:SerfatiForAlla}
	Suppose that $u$ is a solution to the Euler equations in the full plane
	as in \refD{ESol} and that \refE{SerfatiId} holds
	for one, given cutoff function, $a$.
	Then \refE{SerfatiId} holds for any other cutoff function, $b$.
\end{prop}
\begin{proof}
	Let $R_a(t, x)$ be the right-hand side of \refE{SerfatiId} for the cutoff
	function, $a$, and note that it is always finite for any $u$ in $L^\iny(0, T; S)$.
	Letting $h(y) = (a(y) - b(y)) K^j(y)$, $j = 1$ or $2$, $h$ lies in
	$H^2(\R^2)$ and has compact support, so by \refL{gammaLemma},
	\begin{align*}
		R_b(t, &x) - R_a(t, x) \\
			&= - h * (\omega(t) - \omega^0)(x)
				- \int_0^t (\grad \grad^\perp h) \stardot (u \otimes u)(s, x) \, ds
			= 0.	
	\end{align*}
\end{proof}

\begin{lemma}\label{L:gammaLemma}
	Let $h \in H^2(\R^2)$ have compact support.
	Assume that $u$ is a bounded solution to the Euler equations as in \refD{ESol}.
	Then
	\begin{align}\label{e:gammaId}
		h * (\omega(t) - \omega^0)
			= - \int_0^t (\grad \grad^\perp h) \stardot (u \otimes u)(s) \, ds.
	\end{align}
\end{lemma}
\begin{proof}
	Note that the compact support of $h$ gives the finiteness of both convolutions in \refE{gammaId}
	(see \refL{Conv}).	
	Define, for all $\eps$ in $(0, 1/2)$,
	\begin{align*}
		h_\eps(s, x)
			= \phi_\eps(s) h(x),
	\end{align*}
	where $\phi_\eps$ lies in $C^\iny_C((0, t))$ and is chosen so that,
	$\phi_\eps = 1$ on $[\eps, t - \eps]$,
	$\phi_\eps \ge 0$, and
	\begin{align*}
		\phi'_\eps(\cdot)
			\to \delta(\cdot) - \delta(t - \cdot)
			\text{ as } \eps \to 0^+,
	\end{align*}
	the convergence being as Radon measures on $[0, T]$.
	We note, then, that $h_\eps$ lies in $H_0^1((0, t) \times \R^2)$ with compact support
	in $(0, t) \times \R^2$.
	
	Fix $x$ in $\R^2$ and let $B$ be an open ball in $(0, t) \times \R^2$ sufficiently large
	to contain $\supp h_\eps(x - \cdot)$.

	Now, $\grad u \in L^\iny(0, T; L^2_{loc}(\R^2))$ since $\omega \in L^\iny([0, T] \times \R^2)$,
	so $u \cdot \grad u \in L^\iny(0, T; L^2_{loc}(\R^2))$. Thus, $\prt_t \omega
	= -u \cdot \grad \omega = -\curl (u \cdot \grad u) = \dv((u \cdot \grad u)^\perp)$
	lies in $L^\iny(0, T; H_{loc}^{-1}(\R^2))$ and hence in
	$L^\iny(0, T; H^{-1}(B))$.
	Therefore, we have sufficient regularity to apply \refL{Hm1H01} to obtain,
	\begin{align*}
		(\prt_t \omega, \, &h_\eps(x - \cdot))_{H^{-1}(B), H_0^1(B)}
			= (\dv((u \cdot \grad u)^\perp), h_\eps(x - \cdot))_{H^{-1}(B), H_0^1(B)} \\
			&= - \int_0^t \int_{\R^2} (u \cdot \grad u)^\perp(t, y)
				\cdot \grad h_\eps(x - y)) \, dy \, ds \\
			&= \int_0^t \int_{\R^2} (u \cdot \grad u)(t, y)
				\cdot \grad^\perp h_\eps(x - y) \, dy \, ds.
	\end{align*}
		
	Using the vector identity, $(u \cdot \grad u) \cdot V
	    = u \cdot \grad (V \cdot u) - (u \cdot \grad V)
	    \cdot u$
    with $V = \grad^\perp h_\eps(x - \cdot)$ gives
	\begin{align}\label{e:SerfatiIDIBP2}
		\begin{split}
		    \int_{\R^2} (u &\cdot \grad u)(t, y)
				    \cdot \grad^\perp h_\eps(x - y) \, dy
			    =
			    \int_{\R^2} (u \cdot \grad u) \cdot V \\
			&= \int_{\R^2} u \cdot \grad (V \cdot u)
			    - \int_{\R^2} (u \cdot \grad V) \cdot u
			=  - \int_{\R^2} (u \cdot \grad V) \cdot u.
		\end{split}
	\end{align}
	\Ignore{ 
	Observing that $u$ lies in the space $E(B)$ of
	\refL{TemamIBP} and
	$V \cdot u$
	lies in $H^1(B)$, we have sufficient regularity to apply \refL{TemamIBP}
	along with the vector identity, $u \cdot \grad (V \cdot u) = (u \cdot \grad V) \cdot u
	+ (u \cdot \grad u) \cdot V$, to show that
	The one integral vanished in applying \refL{TemamIBP} since $\dv u = 0$ in $\R^2$
	and $u \cdot \n = 0$ on $\prt B$.
	} 
	The one integral vanished because
	$\dv u = 0$ and $\grad (V \cdot u) \in H^1(\R^2)$
	with compact support.
	We conclude from this that
	\begin{align}\label{e:prtomegagameeps1}
		\begin{split}
			(\prt_t \omega, \, &h_\eps(x - \cdot))_
					{H^{-1}(B), H_0^1(B)} \\
				&= - \int_0^t \int_{\R^2} \pr{(u(s, y) \cdot \grad_y) \grad_y^\perp
					h_\eps(x - y)} \cdot u(s, y) \, dy \, ds \\
				&= - \int_0^t \int_{\R^2} (\grad_y \grad_y^\perp h_\eps(x - y))
						\cdot (u \otimes u) (s, y) \, dy \, ds \\
				&\to - \int_0^t (\grad \grad^\perp h) \stardot (u \otimes u)(s, x)
							\, ds 
		\end{split}
	\end{align}
	as $\eps \to 0^+$
	by the dominated convergence theorem.
	
	With $x$ still fixed, let
	\begin{align*}
		f(s)
			= \int_{\R^2} h(x - y) \omega(s, y) \, dy.
	\end{align*}
	Vorticity is transported by the flow map (as in \refD{ESol})
	and $u$ is bounded on $[0, t] \times \R^2$, so
	$f$ is continuous on $[0, t]$. Thus,
	\begin{align}\label{e:prtomegagameeps2}
		\begin{split}
		(\prt_t \omega, \, &h_\eps(\cdot, x - \cdot))
		        _{H^{-1}(B), H_0^1(B)} \\
			&= - (\omega, \prt_t h_\eps(\cdot, x - \cdot))
	                _{H^{-1}(B), H_0^1(B)} \\
			&= - (\omega, \phi_\eps' h(x - \cdot))_{L^2(B)}, \\
			&= - \int_0^t \int_{\R^2} \phi_\eps'(s) h(x - y)
			    \omega(s, y) \, dy \, ds \\
			&= - \int_0^t \phi_\eps'(s) \int_{\R^2} h(x - y)
			    \omega(s, y) \, dy \, ds \\
			&= - \int_0^t \phi_\eps'(s) f(s) ds \to f(t) - f(0)
			    \text{ as } \eps \to 0^+ \\
			&= \int_{\R^2} h(x - y)
			    (\omega(t, y) - \omega^0(y)) \, dy.
		\end{split}
	\end{align}
	
	The identity in \refE{gammaId} follows from \refEAnd{prtomegagameeps1}{prtomegagameeps2}.
\end{proof}

%
%
\subsection{Renormalized Biot-Savart law in the full plane}\label{S:BSLawR2}

\noindent The purpose of this subsection is to prove that for any vector field in $S$, the renormalized Biot-Savart law holds for a subsequence; this is \refP{NearRenormalizedConvergence}.

\begin{prop}\label{P:NearRenormalizedConvergence}
	Assume that $u$ lies in the Serfati space, $S$, of \refD{SerfatiVelocity}.
	Let $\omega = \omega(u)$ and define
	\begin{align*}
		u_R = (a_R K) * \omega.
	\end{align*}
	Then $\omega(u_R) \to \omega(u)$ in $L^\iny$ with
	$\norm{\omega(u_R) - \omega(u)}_{L^\iny} \le C \norm{u}_{L^\iny} R^{-1}$,
	and there exists a subsequence, $(R_k)$, $R_k \to \iny$,
	and a constant vector field, $H$,
	such that $u_{R_k} \to u + H$ as $k \to \iny$ uniformly on compact subsets. 
\end{prop}
\begin{proof}
	First observe that $u_R$ is well-defined as a tempered distribution by \refL{Conv},
	since $a_R K \in \Cal{E}'$.
	Also by that lemma,
	\begin{align*}
		\dv u_R
			= (\dv (a_R K)) * \omega
			= (\grad a_R \cdot K + a_R \dv K) * \omega
			= 0 * \omega
			= 0,
	\end{align*}
	since $\grad a_R \cdot K = 0$, $a_R$ being radially symmetric,
	and $\dv K = 0$.
		
	Then, from \refL{AlmostBSu},
	\begin{align}\label{e:uRCalc}
		\begin{split}
			u_R
				&= \omega(a_R K) * u 
				= (a_R \omega(K) + \grad^\perp a_R \cdot K) * u \\
				&= (a_R \delta + \grad^\perp a_R \cdot K) * u
				= u + (\grad^\perp a_R \cdot K) * u.
		\end{split}
	\end{align}
	But, $(\grad^\perp a_R \cdot K) * u$ is $O(1)$
	by \refL{OrderaRK}, so
	$(u_R)$ is bounded in $L^\iny$.	
	
	Since also $\omega((\grad^\perp a_R \cdot K) * u) = O(R^{-1})$ by \refL{OrderaRK}, we have
	\begin{align*}
		\omega(u_R)
			&= O(R^{-1}) + \omega(u).
	\end{align*}
	We conclude both that $\omega(u_R) \to \omega(u)$ in $L^\iny$ and that $(u_R)$,
	already bounded
	in $L^\iny$, is bounded in $S$.
	
	By \refL{Morrey}, then, $(u_R)$ is an equicontinuous family of pointwise bounded
	functions and hence for any compact subset, $L$, of $\R^2$ some subsequence of $(u_R)$
	converges uniformly on $L$. A diagonalization argument for increasing $L$ gives a
	subsequence, $(u_{R_k})$, that converges
	uniformly on compact subsets to some $\ol{u}$ in $L^\iny$.
	At the same time, as shown above, $\omega(u_R) \to \omega(u)$ and $\dv u_R = 0$.
	
	
	Fix a compact subset, $L$, of $\R^2$ and let $\varphi \in H_0^1(L)$. Then
	\begin{align}\label{e:omegauRConv}
		(\omega&(u_{R_k}), \varphi)
			= - (\dv u_{R_k}^\perp, \varphi)
			= (u_{R_k}^\perp, \grad \varphi)
			\to (\ol{u}^\perp, \grad \varphi)
			= (\omega(\ol{u}), \varphi).
	\end{align}
	But also $(\omega(u_R), \varphi) \to (\omega(u), \varphi)$, so $\omega(\ol{u}) = \omega(u)$ on $L$
	and hence on all of $\R^2$, since $L$ was arbitrary. Similarly, $\dv \ol{u} = \dv u = 0$.
	
	Thus,
	$\dv (u - \ol{u}) = 0$ and $\omega(u - \ol{u}) = 0$. By the identity,
	$ 
		\Delta v
			= \grad \dv v + \grad^\perp \omega(v),
	$
	then,
	$\Delta (u - \ol{u}) = 0$, and we conclude that
	$
		\ol{u} = u + H,
	$
	where $H$ is an harmonic polynomial.
	Since $u$ and $\ol{u}$ lie in $L^\iny$, $H$ must be a constant.
\end{proof}

\begin{lemma}\label{L:OrderaRK}
	Let $\al$, $\beta$ be multi-indices with $\abs{\al} \ge 1$ and $\abs{\beta} \ge 0$. Then
	\begin{align*}
		\smallnorm{D^\al a_R \otimes D^\beta K}_{L^1}
			&\le C R^{1 - \abs{\al} - \abs{\beta}}.
	\end{align*}
	Moreover, if $F \in L^\iny(\R^2)$ then
	\begin{align*}
		\smallnorm{(D^\al a_R \otimes D^\beta K) * F}_{L^\iny}
			\le C \norm{F}_{L^\iny} R^{1 - \abs{\al} - \abs{\beta}}.
	\end{align*}
\end{lemma}
\begin{proof}
	The $L^1$-bound follows because $D^\al a_R$ is supported on an annulus of inner radius,
	$c_1 R$, and outer radius, $c_2 R$, for some $0 < c_1 < c_2$,
	and is bounded by $C R^{-\al}$ on this annulus, while $\abs{\prt_\beta K} \le C R^{-\beta - 1}$
	on this annulus. The bound $(D^\al a_R \otimes D^\beta K) * F$ then follows from
	Young's convolution inequality.
\end{proof}

\begin{lemma}\label{L:stardot}
	For all $f \in \Cal{E}'$, $v \in (\Cal{S}')^2$,
	\begin{align*}
		\grad f \stardot v
			= f * \dv v,
	\end{align*}
	where the $\stardot$ operator is as in \refD{stardot}.
\end{lemma}
\begin{proof}
    Using \refL{Conv},
    \begin{align*}
        \grad f \stardot v
            = \prt_i f *v^i
            = f * \prt_i v^i
            = f * \dv u.
    \end{align*}
\end{proof}

\begin{lemma}\label{L:AlmostBSu}
	For any $u \in S$,
	$
		(a_R K) * \omega(u)
			= \omega(a_R K) * u.
	$
\end{lemma}
\begin{proof}

	We will show that $w := (a_R K) * \omega(u) - \omega(a_R K) * u = 0$.
	We have,
	\begin{align*}
		w^i
			&= (a_R K^i) * (\prt_1 u^2 - \prt_2 u^1)
				- (\prt_1 (a_R K^2) - \prt_2 (a_R K^1)) * u^i \\
			&= \prt_1 (a_R K^i) * u^2 - \prt_2 (a_R K^i) * u^1
				- (\prt_1 (a_R K^2) - \prt_2 (a_R K^1)) * u^i.
	\end{align*}
	Then,
	\begin{align*}
		w^1
			&= \prt_1 (a_R K^1) * u^2 - \prt_2 (a_R K^1) * u^1
				- (\prt_1 (a_R K^2) - \prt_2 (a_R K^1)) * u^1 \\
			&= \prt_1 (a_R K^1) * u^2  - \prt_1 (a_R K^2) * u^1 \\
			&= (\prt_1 a_R K^1) * u^2  - (\prt_1 a_R K^2) * u^1
			    + (a_R \prt_1 K^1) * u^2  - (a_R \prt_1 K^2) * u^1 \\
			&= (\prt_1 a_R K^1) * u^2  - (\prt_1 a_R K^2) * u^1
			        - (a_R \prt_2 K^2) * u^2  - (a_R \prt_1 K^2) * u^1 \\
			&= (\prt_1 a_R K^1) * u^2  - (\prt_1 a_R K^2) * u^1
			        + (\prt_2 a_R K^2) * u^2  + (\prt_ 1 a_R K^2) * u^1 \\
			&\qquad
				- \prt_2 (a_R K^2) * u^2  - \prt_ 1 (a_R K^2) * u^1 \\
			&= (\prt_1 a_R K^1) * u^2 + (\prt_2 a_R K^2) * u^2
			    - \grad (a_R K^2) \stardot u \\
			&= (\grad a_R \cdot K) * u^2
			= 0,
	\end{align*}
	since $\grad a_R \cdot K = 0$, $a_R$ being radially symmetric.
	In the fourth equality we used $\dv K = 0$, and we applied
	\refL{stardot} in the penultimate equality to deduce that
	$\grad (a_R K^2) \stardot u = (a_R K^2) * \dv u = (a_R K^2) * 0 = 0$.
	Similarly,
	\begin{align*}
		w^2
			&= \prt_1 (a_R K^2) * u^2 - \prt_2 (a_R K^2) * u^1
				- (\prt_1 (a_R K^2) - \prt_2 (a_R K^1)) * u^2 \\
			&= - \prt_2 (a_R K^2) * u^1 + \prt_2 (a_R K^1) * u^2 \\
			&= - (\prt_2 a_R K^2) * u^1 + (\prt_2 a_R K^1) * u^2
			    - (a_R \prt_2 K^2) * u^1 + (a_R \prt_2  K^1) * u^2 \\
			&= - (\prt_2 a_R K^2) * u^1 + (\prt_2 a_R K^1) * u^2
			        + (a_R \prt_1 K^1) * u^1 + (a_R \prt_2  K^1) * u^2 \\
			&= - (\prt_2 a_R K^2) * u^1 + (\prt_2 a_R K^1) * u^2
			        - (\prt_1 a_R K^1) * u^1 - (\prt_2 a_R K^1) * u^2 \\
			&\qquad
				 + \prt_1 (a_R K^1) * u^1 + \prt_2 (a_R  K^1) * u^2 \\
			&= - (\prt_2 a_R K^2) * u^1 - (\prt_1 a_R K^1) * u^1
			        + \grad(a_R K^1) \stardot u \\
			&= - (\grad a_R \cdot K) * u^1
			= 0.
	\end{align*}
\end{proof}

\begin{remark}\label{R:RadialConvenientOnly}
	The radial symmetry of $a$ was convenient in the proof of
	\refP{NearRenormalizedConvergence}, but was not essential. Were $a$ not
	radially symmetric, another application of \refL{Conv} would give
	$
		(\grad a_R \cdot K) * \omega
			= (\grad^\perp (\grad a_K \cdot K)) \stardot u.
	$
	This is $O(R^{-1})$ by \refL{OrderaRK} (and the product rule), so $\dv u_R \to 0$
	in $L^\iny(\R^2)$, which yields $\dv \ol{u} = 0$.
	Also, \refL{AlmostBSu} would become $u_R
	= \omega(a_R K) * u - (\grad a_R \cdot K) * u^\perp$, but the extra term
	$(\grad a_R \cdot K) * u^\perp$ can be handled just as
	$(\grad^\perp a_R \cdot K) * u$ is.
\end{remark}

%
%
\subsection{Velocity in the full plane}\label{S:CharacterizationR2}

\noindent We are now in a position to establish our characterization of bounded solutions.

\begin{proof}[\textbf{Proof of \refT{CharacterizationR2} (i, ii)}]
	Suppose that $u$ is a solution to the Euler equations
	as in \refD{ESol}
	and $a$ is any radial cutoff function as in \refD{RadialCutoff}.
	Then from \refP{NearRenormalizedConvergence} there exists
	a subsequence, $(R_k)$, for which
	\begin{align*} 
		\begin{split}
			u(t) - u^0
				= U_\iny(t)
				+ 
					\lim_{k \to \iny} (a_{R_k} K) *
						(\omega(t) - \omega^0)
		\end{split}
	\end{align*}
	for some vector field, $U_\iny(t)$.
	By \refP{RenormalizationLemma} and \refR{RenormalizationLemma},
	the limit then holds for the entire sequence,
	uniformly on compact subsets of $[0, T] \times \R^2$,
	both \refEAnd{SerfatiId}{RenormalizedBS} hold, and $U_\iny \in C([0, T])$.
	Appealing to \refP{NearRenormalizedConvergence} once more, we see that the limit in
	\refE{RenormalizedBS} holds locally in $S$ (in fact, the vorticities converge in
	$L^\iny(\R^2)$).
	By \refP{SerfatiForAlla}, $U_\iny$ is independent of the choice of cutoff function, $a$.
	
	It then follows from \refE{SerfatiId}, the transport of the vorticity by the flow map, the boundedness
	of the velocity, the absolute continuity of the integral, the continuity of $u$ in $L^\iny([0, T])$,
	and the continuity of $U_\iny$,
	that $U_\iny(0) = 0$.
\end{proof}

\Ignore{ 
\begin{remark}\label{R:RenormalizedBSLawRemark}
	It follows from
	\refPAnd{NearRenormalizedConvergence}{NearRenormalizedConvergenceExt}
	that there exists a subsequence,
	$(R_k)$, such that the renormalized Biot-Savart law,
	\begin{align}\label{e:u0Limit}
		u^0(x)
			= H + 
					\lim_{k \to \iny} \int_{\R^2} a_{R_k}(x - y) K(x - y)
						\omega^0(y) \, dy,
	\end{align}
	holds for some $H$. By \refE{RenormalizedBS},
	it must therefore be true that for this same subsequence,
	\begin{align}\label{e:utxLimit}
		u(t, x)
			= - U_\iny(t) + H + 
					\lim_{k \to \iny} \int_{\R^2} a_{R_k}(x - y) K(x - y)
						\omega(t, y) \, dy.
	\end{align}
	If \refE{u0Limit} holds for distinct values of $H$ on two subsequences,
	then by absolute continuity of the integral there is a subsequence on which
	it holds for any intermediate
	value. This reflects ripples of increasingly large scale in the initial vorticity, and \refE{utxLimit}
	says that these ripples persist over time. Explicit examples of initial vorticities for
	which $H$ can vary are not known.

\end{remark}
} 

%
%
\section{Existence and uniqueness in the full plane}\label{S:ExistenceR2}

\noindent Our proof of \refT{ExistenceR2} begins with the following lemma:

\begin{lemma}\label{L:COVEquiv}
    Let $(u, p)$ and $(\ol{u}, \ol{p})$ be related as in the transformation,
    \refE{COV}. Then $(u, p)$ satisfy \refE{EClassical} if and only if
    $(\ol{u}, \ol{p})$ satisfy \refE{EClassical}.
    Moreover, $u$ is a bounded solution to the Euler equations as in
    \refD{ESol} if and only if $\ol{u}$ is such a solution.
\end{lemma}
\begin{proof}
    Applying the chain rule gives,
    \begin{align*}
        \prt_t \ol{u}(t, x)
            &= \prt_t u(t, \ol{x}) + U_\iny(t) \cdot
                \grad u(t, \ol{x})
                - U_\iny'(t), \\
        \grad \ol{u}(t, x)
            &= \grad u(t, \ol{x}), \\
        \grad \ol{p}(t, x)
            &= \grad p(t, x) + U_\iny'(t), \\
        \dv \ol{u}(t, x)
            &= \dv u(t, \ol{x}),
    \end{align*}
    from which it follows that
    \begin{align*}
        \prt_t \ol{u}(&t, x) + \ol{u}(t, x) \cdot \grad \ol{u}(t, x)
                + \grad \ol{p}(t, x) \\
            &= \prt_t u(t, \ol{x})
                + u(t, \ol{x}) \cdot \grad u(t, \ol{x})
                + \grad p(t, \ol{x}).
    \end{align*}
    Thus, $(u, p)$ satisfies \refE{EClassical} if and only if
    $(\ol{u}, \ol{p})$ satisfies \refE{EClassical} (since
    $U_\iny(0) = 0$).
    
    Let $\ol{\omega} = \curl \ol{u}$. Then the chain rule gives
    \begin{align*}
        \ol{\omega}(t, x)
            & = \omega(t, \ol{x}), \\
        \prt_t \ol{\omega}(t, x)
            &= \prt_t \omega(t, \ol{x})
                + \prt_t \ol{x} \cdot \grad \omega(t, \ol{x}) \\
            &= \prt_t \omega(t, \ol{x})
                + U_\iny(t) \cdot \grad \omega(t, \ol{x}), \\
        \grad \ol{\omega}(t, x)
            &= \grad \omega(t, \ol{x}), 
    \end{align*}
    from which it follows that
    \begin{align*}
        \prt_t \ol{\omega}(t, x)
                + \ol{u}(t, x) \cdot \grad \ol{\omega}(t, x)
            = \prt_t \omega(t, \ol{x})
                    + u(t, \ol{x}) \cdot \grad \omega(t, \ol{x}).
    \end{align*}
    Hence, the vorticity equation of the Euler equations is satisfied
    in \refD{ESol} for $u$ if and only if it is satisfied for $\ol{u}$.
    
    Let $X$, $Y$ be the flow maps for $u$, $\ol{u}$, respectively.
    The flow maps are related by the identity, $X(t, x) = \ol{Y(t, x)}$,
    since then
    \begin{align*}
        \prt_t X(t, x)
            &= \prt_t \pr{Y(t, x) + \int_0^t U_\iny(s) \, ds}
            = \ol{u}(t, Y(t, x)) + U_\iny(t) \\
            &= u(t, \ol{Y(t, x)})
            = u(t, X(t, x)).
    \end{align*}
    Thus, $\omega(t, X(t, x)) = \omega^0(x)$ for all $t$, $x$ if and only if
    $\ol{\omega}(t, Y(t, x)) = \omega^0(x)$ for all $t$, $x$ since
    $\ol{\omega}(t, Y(t, x)) = \omega(t, \ol{Y(t, x)}) = \omega(t, X(t, x))$.
\end{proof}

\begin{proof}[\textbf{Proof of \refT{ExistenceR2}}]
    Assume that $u^0 \in S$, let $T > 0$ be arbitrary, and fix
    $U_\iny \in (C[0, T])^2$ with $U_\iny(0) = 0$.
    Let $\ol{u}^0 = u^0 - U_\iny(0) = u^0$, and
    let $\ol{u}$
    be the Serfati solution with initial velocity $\ol{u}^0$
    constructed in \cite{AKLL2014}. Then, as shown in
    \cite{AKLL2014},
    $\ol{u}$ is the unique bounded solution satisfying
    \ThmProp{i} of \refT{CharacterizationR2}
    with $U_\iny \equiv 0$. By \refT{CharacterizationR2},
    \ThmProp{ii} is equivalent to \ThmProp{i}, and so also holds.
    Making the inverse change of
    variables from that in \refE{COV} then yields a
    bounded solution, $(u, p)$, satisfying \ThmProp{i} and
    \ThmProp{ii} with the original $U_\iny$.
    This also gives uniqueness criteria (a) and (b).
    
    That \ThmPropThrough{iii}{v} hold for $(u, p)$ will be shown
    when we establish
    the properties of the pressure in \refS{PressureR2}.

    Uniqueness criteria (c) is proved, for $U_\iny \equiv 0$,
    in \cite{TaniuchiEtAl2010}, and it can also be adapted
    to a nonzero $U_\iny$ using
    the change of variables in \refE{COV}. Finally,
    we observe that uniqueness
    criteria (d) immediately implies (c).
\end{proof}

\begin{remark}\label{R:uSBound}
    The solution, $\ol{u}$, constructed in \cite{AKLL2014}
    (and hence, by uniqueness, any such solution)
    also has the property that
    \begin{align*}
        \norm{\ol{u}(t)}_{L^\iny}
            \le e^{C(1 + \smallnorm{\omega^0}_{L^\iny}) t} \smallnorm{u^0}_{L^\iny}.
    \end{align*}
    Also, $\norm{\omega(\ol{u})(t)}_{L^\iny} = \smallnorm{\omega^0}_{L^\iny}$,
    since vorticity is transported by the flow map. Hence,
    \begin{align*}
        \norm{\ol{u}(t)}_S
            \le e^{C(1 + \smallnorm{\omega^0}_{L^\iny}) t} \smallnorm{u^0}_S.
    \end{align*}
    Then, since $\norm{u(t)}_S = \norm{\ol{u}(t) - U_\iny(t)}_S \le \norm{\ol{u}(t)}_S
    + \norm{U_\iny(t)}$, we have
    \begin{align}\label{e:uSBound}
        \norm{u(t)}_S
            \le C_S(t) \smallnorm{u^0}_S + \norm{U_\iny(t)},
            \text{ where }
            C_S(t) = e^{C(1 + \smallnorm{\omega^0}_{L^\iny}) t}.
    \end{align}
\end{remark}

The convenient transformation in \refE{COV} allowed us to simply use the existence and uniqueness theorem of \cite{AKLL2014}, avoiding the need to modify its proof to accommodate $U_\iny \not \equiv 0$. To establish the properties of the pressure in \refT{CharacterizationR2}, however, we need the approximate sequence of smooth velocities, $(u_n)$, used in \cite{AKLL2014} to obtain existence of a solution. Adjusting the sequence in \cite{AKLL2014} to accommodate $U_\iny$ by employing a sequence, $(U^n_\iny)$, converging to $U_\iny$ leads to a sequence, $(u_n)$, of approximate classical solutions with the following properties:
\begin{align}\label{e:unConvForPressureR2}
	\begin{split}
		&(u_n) \text{ is bounded in } C([0, T] \times S), \\
		&u_n \to u \text{ uniformly on compact subsets of }
		    [0, T] \times \R^2, \\
		&\omega(u_n) \to \omega(u) \text{ in } L^p_{loc}(\R^2)
			\text{ for all $p$ in } [1, \iny), \\
		&u_n(t, x) = U^n_\iny(t) + O(\abs{x}^{-1}), \\
		&U_\iny^n \to U_\iny \text{ in } C([0, T]), \\
		&(U_\iny^n)' \to U_\iny' \text{ in } \Cal{D}'((0, T)).
	\end{split}
\end{align}
We will use these properties in \refS{PressureR2}.

%
%
\section{The pressure in the full plane}\label{S:PressureR2}

\noindent In this section, we characterize the pressure for solutions to the 2D Euler equations in the full plane as in $\refE{CharacterizationR2}_{2, 3}$, stated more precisely as properties \ThmProp{iii}-\ThmProp{v} of \refT{CharacterizationR2}.

To understand the difficulties in characterizing the asymptotic behavior of the pressure at infinity, consider first the simpler case of a smooth solution, $u$, to the Euler equations having compactly supported vorticity with $u$ vanishing at infinity. In such a case, $u$ decays like $C \abs{x}^{-1}$ at infinity, while $\grad u$ decays like $C \abs{x}^{-2}$ (as in \refL{graduDecayR2}).

Taking the divergence of $\prt_t u + u \cdot \grad u + \grad p = 0$, we see that $p$ is a solution to $\Delta p = - \dv(u \cdot \grad u) = - \dv \dv (u \otimes u)$. A particular solution is given by $q = R (u \otimes u)$ for the (multiple) Riesz transform, $R = - \Delta^{-1} \dv \dv$. Any other solution differs from $q$ by an harmonic polynomial, $h(t)$, so $p = h + q$.

The decay of $u$ gives $u \otimes u \in L^r(\R^2)$ for all $r \in (1, \iny]$. By the Calder\'{o}n-Zygmund theory, then, $q \in L^r(\R^2)$ for all $r \in (1, \iny)$, so it decays at infinity. Moreover, $\grad q = T (u \cdot \grad u)$, where $T = - \Delta^{-1} \grad \dv$ is also a singular integral operator of Calder\'{o}n-Zygmund type. From the decay of $u \cdot \grad u$ follows the decay of $\grad q$ at infinity. Then the decay, after integrating in time, of $\prt_t u + u \cdot \grad u$ at infinity forces $h$ to be constant in space. We conclude that there exists a unique pressure decaying at infinity.

Now let $u$ be a bounded solution to the Euler equations of \refD{ESol}. We can still obtain a particular solution, $q = R (u \otimes u)$, to $\Delta p = - \dv \dv (u \otimes u)$ using the above argument because $R$ maps $L^\iny$ into $BMO$, and $u \otimes u \in L^\iny$. A bound on the growth of $q$ at infinity could also be obtained formally by applying \refP{PressureConvR2} (this lemma is at the heart of the matter), and rigorously by making a simple approximation argument. Then, arguing as above, we can conclude that \textit{if} a valid pressure exists then it differs from $q$ by an harmonic polynomial, $h$.

To determine, $h$, however, we would need to understand the behavior at infinity of $\prt_t u + u \cdot \grad u$ (at least integrated over time) to obtain a pressure $p = q + h$ satisfying $\prt_t u + u \cdot \grad u + \grad p = 0$. But even the behavior of $u$ at infinity is defined only in the weak sense of $\refE{CharacterizationR2}_1$; it appears to be impossible to say anything useful about the behavior of $\prt_t u + u \cdot \grad u$ at infinity.

These difficulties naturally lead us to the idea of using an approximate sequence of vector fields, $(u_n)$, decaying sufficiently rapidly at infinity and converging in an appropriate sense to $u$. We could construct such a sequence in an ad hoc manner, but we already have such a sequence at hand: the sequence of approximate solutions with the properties given in \refE{unConvForPressureR2}. This sequence has the virtue that the approach we described above for obtaining a pressure applies to it (after making the transformation in \refE{COV}), so there exists a corresponding sequence of pressures, $(p_n)$, for which $\prt_t u_n + u_n \cdot \grad u_n + \grad p_n = 0$. We will show that this sequence of pressures converges to our desired pressure.

Our proof of \ThmProp{iii}-\ThmProp{v} of \refT{CharacterizationR2} begins by proving \refPThrough{SmoothPressureR2}{gradpBoundR2}, which establish properties of the pressure for the approximate solutions, $(u_n)$, of \refE{unConvForPressureR2}. Once we establish these properties, it will remain only to make an approximation argument to establish the existence of a pressure, $p$, for the velocity, $u$, having the same properties as the approximate sequence of pressures.


Our first proposition provides an explicit expression for the pressure, $p_n$:
\begin{prop}\label{P:SmoothPressureR2}
    Let $G(x) = (2 \pi)^{-1} \log \abs{x}$, the fundamental solution to
    the Laplacian in $\R^2$. Let
    \begin{align}\label{e:pqnR2}
	     \begin{split}
			    q_n(t, x)
				    &= a_n(t)
				        - G*\dv \dv (u_n(t) \otimes u_n(t))(x), \\
			    p_n(t, x)
				    &= -(U_\iny^n)'(t) \cdot x
					        + q_n(t, x),
	    \end{split}
    \end{align}
where $a_n(t)$ is chosen so that $p_n(t, 0) = q_n(t, 0) = 0$ for all $t$. Then $\prt_t u_n + u_n \cdot \grad u_n + \grad p_n = 0$.
\end{prop}
\begin{proof}
    This result for $U_\iny^n \equiv 0$ is classical (the argument being that given at the
    beginning of this section). For nonzero $U_\iny^n$, we simply use the transformation
    in \refE{COV} and apply the first part of \refL{COVEquiv}.
\end{proof}

Our second proposition bounds the growth of $p_n$ (less the harmonic part) at infinity:
\begin{prop}\label{P:PressureConvR2}
	Let $q_n$ be given by $\refE{pqnR2}_1$. Then,
	\begin{align*}
        \abs{q_n(t, x)}
	        \le C C_S(t) \smallnorm{u^0}_S^2 \log (e + \abs{x})
    \end{align*}
	for some absolute constant $C$ (in particular, independent of $n$), where $C_S(t)$
	is given in \refE{uSBound}.
	Also, $q_n$ has a bound on its
	log-Lipschitz norm uniform over $[0, T]$ that is independent of $n$.
\end{prop}
\begin{proof}
We can write $q_n = a_n(t) - R \, h_n$, where $h_n = u_n \otimes u_n$ and $R = \Delta^{-1} \dv \dv$ is a Riesz transform. Here, $\Delta^{-1} f = -\FTR(\abs{\cdot}^2 \wh{f})$, $\FTR$ being the inverse Fourier transform.
Observe that $h_n \in LL$ with $\norm{h_n(t)}_{LL} \le C \norm{u(t)}_{S}^2 \le C_S(t)^2 \smallnorm{u^0}_S^2$ by \refL{Morrey} and \refE{uSBound}. The result then follows from \refL{pDecay}.
\end{proof}

Our third proposition give an expression for $\grad p_n$ analogous to \refE{gradpR2} and shows that it is bounded:
\begin{prop}\label{P:gradpBoundR2}
    The identity,
    \begin{align}\label{e:gradpExp}
    \begin{split}
        \grad p_n&(x)
			    = -(U_\iny^n)' \\
			    &
				    + \int_{\R^2} a(x - y) K^\perp(x - y)
				        \dv \dv (u_n \otimes u_n)(y) \, dy \\
			    &
	 			    + \int_{\R^2} (u_n \otimes u_n)(y) \cdot
					    \grad_y \grad_y \brac{(1 - a(x - y)) K^\perp(x - y)}
					    \, dy,
	    \end{split}
    \end{align}
    holds independently of the choice of cutoff function,
    and $\grad p_n + (U_\iny^n)'$ is bounded uniformly in
    $L^\iny([0, T] \times \R^2)$.
\end{prop}
\begin{proof}
Taking the gradient of $p_n$ as given in \refE{pqnR2}, we have
\begin{align*}
	\grad p_n(t, x)
		&= - V_n(t)
			- \int_{\R^2} \grad_x G(x - y) \dv(u_n \cdot \grad u_n)(t, y) \, dy,
\end{align*}
where $V_n = (U_\iny^n)'$.

For $i = 1, 2$ let $j = 2, 1$. Then since $- \grad_x G(x - y) = K^\perp(x - y)$, we can write 
\begin{align*}
	(-1)^i \prt_i p_n(x) + (-1)^i V_n^i
		=  \int_{\R^2} K^j(x - y) \dv(u_n \cdot \grad u_n)(y) \, dy.
\end{align*}
Here, we suppress the time variable to streamline notation.
Applying a cutoff and integrating by parts,
\begin{align*}
	(-1)^i &\prt_i p_n(x) + (-1)^i V_n^i \\
		&= \int_{\R^2} a(x - y) K^j(x - y) \dv(u_n \cdot \grad u_n)(y) \, dy \\
		&\qquad
			+ \int_{\R^2} (1 - a(x - y)) K^j(x - y)
			        \dv(u_n \cdot \grad u_n)(y) \, dy \\
		&= \int_{\R^2} a(x - y) K^j(x - y) \dv(u_n \cdot \grad u_n)(y) \, dy \\
		&\qquad
	 		- \int_{\R^2} (u_n \cdot \grad u_n)(y)
			\cdot \grad \brac{(1 - a(x - y)) K^j(x - y)} \, dy.
\end{align*}
Integrating as in \refE{SerfatiIDIBP2} gives
\begin{align*}
	\begin{split}
	\prt_i &p_n(x) + V_n \\
		&= (-1)^i \int_{\R^2} a(x - y) K^j(x - y)
		    \dv(u_n \cdot \grad u_n)(y) \, dy \\
		&\quad
	 		+ (-1)^i \int_{\R^2} (u_n(y) \cdot \grad_y)
			\grad_y \brac{(1 - a(x - y)) K^j(x - y)}
			\cdot u_n(y) \, dy,
	\end{split}
\end{align*}
which we can write more succinctly as \refE{gradpExp}.

Letting $q$ be \Holder conjugate to $p$ with $p$ in $(1, 2)$, we conclude, since $\dv (u_n \cdot \grad u_n) = \grad u_n \cdot (\grad u_n)^T$, that
\begin{align*} 
	\begin{split}
	\norm{\prt_i p_n + (U_\iny^n)'}_{L^\iny}
		&\le \smallnorm{a K}_{L^p}
			\norm{\grad u_n}_{L^{2q}(\supp a(x - \cdot))}^2 \\
		&
			+ \smallnorm{\grad_y \grad_y \brac{(1 - a) K^j}}_{L^1_y}
				\norm{u_n}_{L^\iny}^2.
	\end{split}
\end{align*}
But by \refL{Morrey},
$
	\norm{\grad u_n}_{L^{2q}(\supp a(x - \cdot))}
		\le C \smallnorm{u^0_n}_S
		\le C \smallnorm{u^0}_S.
$
Given the uniform bound on $u_n$ in $S$ it follows from \refEAnd{D2KKBound}{RearrangementBound}
that $\grad p_n + (U_\iny^n)'$ lies in $L^\iny([0, T] \times \R^2)$ with a bound that is independent of $n$.

It is easy to verify that the expression in \refE{gradpExp} is independent of the choice of cutoff function, $a$, by subtracting the expression for two different cutoffs then undoing the integrations by parts. (That \refE{gradpR2} is independent of the choice of cutoff function follows the same way.)
\end{proof}


\begin{proof}[\textbf{Proof of \ThmProp{iii}-\ThmProp{v} of \refT{CharacterizationR2}}] $\\$
Recall that the sequence $(u_n)$ has the properties in \refE{unConvForPressureR2}.
Let $p_n$ and $q_n$ be as in \refP{SmoothPressureR2}.
By \refP{gradpBoundR2}, $(q_n)$ is an equicontinuous family on $[0, T] \times \R^2$, so it follows, via Arzela-Ascoli and a simple diagonalization argument applied to an increasing sequence of compact subsets of $\R^2$, that a subsequence of $(q_n)$, which we relabel to use the same indices, converges uniformly on compact subsets, and hence as distributions, to some scalar field, $\ol{q}$. Letting $\ol{p} = - U_\iny' \cdot x + \ol{q}$, it follows that $p_n \to \ol{p}$ in $\Cal{D}'((0, T) \times \R^2)$ and also that $\ol{p}(t, 0) = 0$ for all $t$.

From $\refE{unConvForPressureR2}_{1, 2, 3}$ it follows that $\prt_t u_n \to \prt_t u$ and $u_n \cdot \grad u_n \to u \cdot \grad u$ in $\Cal{D}'((0, T) \times \R^2)$. But $\grad p_n \to \grad \ol{p}$ in $\Cal{D}'((0, T) \times \R^2)$ and by \refP{SmoothPressureR2}, $\prt_t u_n + u_n \cdot \grad u_n + \grad p_n = 0$, so $\prt_t u + u \cdot \grad u + \grad \ol{p} = 0$. Thus, $\ol{p}$ is a valid pressure field, so we can use $p = \ol{p}$.

Because $p_n \to p$ uniformly on compact subsets, \refE{pBMO} holds and the bound on $p_n + (U_\iny^n)'$ in \refP{PressureConvR2} yields \refE{pBoundR2}. That \refE{pRieszRel} holds follows from Theorem 2 item (1) of \cite{JunKato2003}.

We complete the proof by establishing that \refE{gradpR2} holds for $p$ and that $\grad p + U_\iny' \in L^\iny([0, T] \times \R^2)$.

Let $\Pi$ be the expression on the right-hand side of \refE{gradpR2}. We will show that $\grad p_n + U_n' \to \Pi + U'$ in $L^\iny([0, T] \times \R^2)$ and hence $\grad p_n \to \Pi$ in $\Cal{D}'((0, T) \times \R^2)$. But we already know that $p_n \to p$ in $\Cal{D}'((0, T) \times \R^2)$ so $\grad p_n \to \grad p$ in $\Cal{D}'((0, T) \times \R^2)$. We can then conclude that $\Pi = \grad p$, that \refE{gradpR2} holds, and that $\grad p + U' \in L^\iny([0, T] \times \R^2)$.

We now show that $\grad p_n + U_n' \to \Pi + U'$ in $L^\iny([0, T] \times \R^2)$.

We write \refE{gradpExp} with $a$ replaced by $a_\eps$, where $\eps$ is to be determined:
\begin{align*}
		\grad p_n&(t, x)
			= -(U_\iny^n)'(t)
				+ \int_{\R^2} a_\eps(x - y) K^\perp(x - y)
				    \dv \dv (u_n \otimes u_n)(t, y) \, dy \\
			&
	 			+ \int_{\R^2} (u_n \otimes u_n)(t, y) \cdot
					\grad_y \grad_y \brac{(1 - a_\eps(x - y))
					    K^\perp(x - y)}
					\, dy \\
		&=: -(U_\iny^n)'(t) + I_1^n(\eps) + I_2^n(\eps).
\end{align*}
The value of $\grad p_n$ is independent of our choice of $\eps$, since, by \refP{gradpBoundR2}, it is independent of the cutoff function $a_\eps$.
Let  $I_1(\eps)$, $I_2(\eps)$ be the corresponding integrals on the right-hand side of \refE{gradpR2}.

Let $\delta > 0$, fix $p$ in $(1, 2)$, and let $q$ be {\Holder} conjugate to $p$. By \refL{Morrey},
\begin{align*}
	\norm{\grad u}_{L^{2q}(\supp a_\eps(x - \cdot))}
		\le C \eps^{\frac{1}{q}}
		\smallnorm{u^0}_S
		\le C \eps^{\frac{1}{q}}.
\end{align*}
Because $\dv \dv (u \otimes u) = \grad u \cdot (\grad u)^T$, this bound gives
\begin{align*}
	\norm{\dv \dv (u \otimes u)}_{L^q(\supp a_\eps(x - \cdot))}
		\le C \eps^{\frac{2}{q}}.
\end{align*}
Since $\abs{K(x)} = C \abs{x}^{-1}$, \Holders inequality gives
\begin{align*}
    \norm{I_1(\eps)}_{L^\iny}
        \le C \eps^{\frac{2}{p} - 1 + \frac{2}{q}}
        = C \eps
\end{align*}
and, similarly, $\norm{I_1^n(\eps)}_{L^\iny} \le C \eps$ uniformly for all $n$.  Choose $\eps = \delta/(3 C)$ so that $C \eps < \delta/3$. Because $u_n \to u$ uniformly on compact subsets of $([0, T] \times \R^2)$, there exists $N > 0$ such that $n > N \implies \norm{I_2(\eps) - I_2^n(\eps)}_{L^\iny} < \delta/3$. (We also use the uniform boundedness of $(u_n)$ to control the tails of the integrals in $I_2(\eps)$, $I_2^n(\eps)$.) Since the value of $\grad p_n$ is independent of $\eps$, this shows that for all $n > N$,
\begin{align*}
	&\norm{\grad p_n + (U_\iny^n)' - \Pi - U'}_{L^\iny} \\
	    &\qquad
		\le
		 \norm{I_1(\eps)}_{L^\iny} + \norm{I_1^n(\eps)}_{L^\iny}
			+ \norm{I_2^n(\eps) - I_2(\eps)}_{L^\iny}
		< \delta.
\end{align*}
These bounds are uniform in time and in space; hence, $\grad p_n + (U_\iny^n)' \to \Pi + (U_\iny)'$ in $L^\iny([0, T] \times \R^2)$. Thus, $\grad p_n \to \Pi$ in $\Cal{D}'((0, T) \times \R^2)$, since $(U_\iny^n)' \to U'$ in $\Cal{D}'((0, T))$.

We now have that \refEThrough{gradpR2}{pBoundR2} hold, $\grad p + U' \in L^\iny([0, T] \times \R^2)$, and $\prt_t u + u \cdot \grad u + \grad p = 0$, which completes the proof.
\end{proof}

\begin{remark}
The log-Lipschitz MOC that we obtained in \refP{PressureConvR2} is a side effect of the manner of proof: it is not as strong as the Lipschitz MOC we obtain in \refP{gradpBoundR2}, though that proposition does not establish decay of $p_n$.
\end{remark}

\begin{lemma}\label{L:graduDecayR2}
    For any $n$ there exists a constant, $C > 0$, such that
    \begin{align*} 
        \abs{u_n(\cdot, x) - U_\iny(\cdot)}_{L^\iny([0, T])}
            &\le \frac{C}{(1 + \abs{x})}, \\
        \abs{\grad u_n(\cdot, x)}_{L^\iny([0, T])}
            &\le \frac{C}{(1 + \abs{x})^2}.
    \end{align*}
\end{lemma}
\begin{proof}
    Because $\omega_n$ is compactly supported there is some $R > 0$ such
    that $\supp \omega_n \subseteq B_R(0)$. Let $\abs{x} > 2R$. Then 
    because $u_n$ is smooth, we have
    \begin{align*}
          \grad u_n(x)
              &= (\grad K) * \omega_n(x)
              = \int_{B_R(0)} \grad_x K(x - y) \omega_n(y) \, dy,
    \end{align*}
    noting that the compact support of $\omega$ eliminates the singularity
    in $\grad_x K(x - y)$.
    But  for all $y \in B_R(0)$,
    \begin{align*}
        \abs{\grad_x K(x - y)}
            &\le \frac{1}{2 \pi (\abs{x} - R)^2}
            \le \frac{1}{2 \pi (\abs{x}/2)^2}
            \le \frac{2}{\pi \abs{x}^2}
    \end{align*}
    
    so
    \begin{align*}
          \abs{\grad u_n(x)}
              &\le \frac{2}{\pi \abs{x}^2} \int_{B_R(0)}
              \abs{\omega_n(y)} \, dy
              = \frac{2}{\pi \abs{x}^2} \norm{\omega_n}_{L^1}.
    \end{align*}
    Since $u_n$ is smooth, $\grad u_n$ is bounded on $B_{2R}(0)$.
    The bound on $\grad u_n$ follows. The bound on $u_n$
    is obtained similarly.
\end{proof}

%
%
\section{The Poisson problem in the full plane}\label{S:PoissonFullPlane}

\noindent In \refS{PressureR2}, we needed to solve the Poisson problem to obtain the pressure in the full plane, our interest being in obtaining the asymptotic behavior of the pressure at infinity. Fortunately, a tool, \refL{ThreeRussians}, for obtaining the MOC of the pressure expressed in terms of a Riesz transform exists in the literature, and we can use it to obtain this asymptotic behavior. As applied in \refS{PressureR2}, we do this for the sequence of approximating solutions, which have sufficient decay at infinity so that the Riesz transforms exist in the classical sense of principal values of singular integrals.

\begin{lemma}\label{L:ThreeRussians}
	Let $R$ be any Riesz transform in $\R^2$. Suppose that $h$
	lying in $L^p(\R^2)$ for some $p$ in $[1, \iny)$ has a concave
	Dini MOC, $\mu$, as in \refD{Dini}. Then
	$R h$ has a MOC, $\nu$, given by
	\begin{align}\label{e:nuMOC}
		\nu(r) = C \pr{S_\mu(r) + r \int_r^\iny \frac{\mu(s)}{s^2} \, ds}
	\end{align}
	for some absolute constant, $C$. (Note that this MOC holds for all
	$r > 0$.)
\end{lemma}
\begin{proof}
	This type of bound in dimension higher than one appears to have been
	first proven by Charles Burch in \cite{Burch1978}
	for a bounded domain (though the MOC he obtains applies only away from
	the boundary
	and $r$ must be sufficiently small).
	It is proved in the whole plane in \cite{KNV2007}.
	\Ignore{ 
	In that reference, the requirement that
	$h$ lie in $L^p(\R^2)$ is not explicitly stated. The proof, however,
	requires that the integral
	representation of $R h$ be absolutely convergent, so some such requirement
	must be imposed.
	} 
\end{proof}

\Ignore{ 
\begin{cor}\label{C:ThreeRussians}
	The assumption in \refL{ThreeRussians} that $f$ lie in $L^p(\R^2)$ for some $p$ in $[1, \iny)$
	can be dropped.
\end{cor}
\begin{proof}
	Let $\varphi$ be a smooth cutoff function supported on $B_2(0)$ and identically equal
	to 1 on $B_1(0)$, and let
	$\varphi_n(\cdot) = \varphi(\cdot/n)$. Defining $f_n = \varphi_n f$, we see that
	each $f_n$ lies in $L^p(\R^2)$ for all $p$ in $[1, \iny]$ and that
	for some constant, $\al > 0$, each $f_n$ shares the same MOC, $\al \mu$.
	Moreover, the sequence, $(f_n)$ converges as distributions to $f$.
	
	Since $f_n \to f$ as distributions, $R f_n \to Rf$ as distributions.
	\textbf{Is this true?}
	But then by \refL{ThreeRussians}, the sequence $(R f_n)$ shares
	a common MOC, $\al \nu$. Thus, they form an equicontinuous family and so a subsequence
	converges (uniformly) on any compact subset to a function, $g$, having the same MOC,
	$\al \nu$. But by the uniqueness of distributional limits, it must be that $g = R f$ and that
	the whole sequence converges to $g$.
\end{proof}
} 

The following corollary of \refL{ThreeRussians} (though not its proof) is inspired by Lemma 2 of \cite{Serfati1995B}.
\begin{lemma}\label{L:pDecay}
    Let $R$ be a Riesz transform
    and assume that $h$ is a tensor field in $LL(\R^2) \cap L^p(\R^2)$
    for some $p$ in $[1, \iny)$. Let $q = R h$.
    Then $q$ is uniformly continuous with the MOC,
    $\nu(s) = C \norm{h}_{LL} s (\log s)^2$,
    for all sufficiently small $s > 0$, and $\abs{q(x) - q(0)} \le
    C \norm{h}_{LL} \log(e +\abs{x})$, for some $C > 0$.
\end{lemma}
\begin{proof}
Referring to \refE{muLL}, since $h$ is bounded and has a log-Lipschitz MOC, we have $\abs{h(x) - h(x + y)} \le \mu(\abs{y})$, where
\begin{align*}
	\mu(r)
		= \left\{
			\begin{array}{rl}
				-M r \log r, & \text{if } \abs{r} \le e^{-1}, \\
				M e^{-1}, & \text{if } \abs{r} > e^{-1},
			\end{array}
		\right.
\end{align*}
where $M = \norm{h}_{LL}$.
Thus, when $r \le e^{-1}$,
\begin{align*}
	S_\mu(r)
		= -M \int_0^r \log s \, ds
		= M (r - r \log r).
\end{align*}
Noting that $S_\mu(e^{-1}) = M e^{-1}$, when $r > e^{-1}$, we have
\begin{align*}
	S_\mu(r)
		&= S_\mu(e^{-1}) + \int_{e^{-1}}^r \frac{M e^{-1}}{s} \, ds
		= M e^{-1} + M e^{-1} (\log r - \log e^{-1}).
\end{align*}

Further, when $r > e^{-1}$,
\begin{align*}
	r \int_r^\iny \frac{\mu(s)}{s^2} \, ds
		&= r \int_r^\iny \frac{M e^{-1} ds}{s^2}
		= M e^{-1}\frac{r}{r}
		= M e^{-1}.
\end{align*}
and when $r < e^{-1}$,
\begin{align*}
	r \int_r^\iny &\frac{\mu(s)}{s^2} \, ds
		= - r \int_r^{e^{-1}} \frac{M \log s}{s} \, ds
			+ r \int_{e^{-1}}^\iny \frac{M ds}{s^2} \\
		&= -M r \frac{1}{2} \brac{(\log s)^2}_r^{e^{-1}} + M r e^{-1}
		= \frac{M}{2} r \brac{1 + (\log r)^2} + M r e^{-1}.
\end{align*}

Applying \refL{ThreeRussians}, then, for $r > e^{-1}$,
\begin{align}\label{e:nuLargeArg}
	\nu(r)
		= C M \pr{\log r + 1}
\end{align}
while for $r \le e^{-1}$,
\begin{align*}
	\nu(r)
		&= C M r \brac{- \log r + (\log r)^2},
\end{align*}
which gives the MOC for $q$ for small argument.
\end{proof}

\begin{remark}
    As we can see from the proof of \refL{pDecay}, the logarithmic
    bound on the growth of $q$ at infinity comes from the
    $L^\iny$-norm of $h$ plus $S_\mu(e^{-1})$. Thus, such
    a logarithmic bound would hold for \textit{any} $h$ in
    $L^\iny(\R^2)$ as long as it also has \textit{some} Dini
    MOC. Note, however,
    that $h \in L^\iny$, which would imply $q \in BMO$, is not
    by itself sufficient to obtain such a bound.
\end{remark}

%
%
\section{Afterword}\label{S:Afterword}

\noindent We have characterized the behavior at infinity of 2D bounded solutions to the Euler equations in the full plane, including properties of the velocity and pressure, and have proved their existence and uniqueness. In the subsections that follow, we make three further observations: The first concerns a vorticity formulation of weak solutions; the second concerns the relation between our results and those of Taniuchi in \cite{Taniuchi2004} and Taniuchi, Tashiro, and Yoneda in \cite{TaniuchiEtAl2010}; the third concerns an extension of these results to the exterior of a single obstacle.

\subsection{Vorticity formulation of weak Solutions}\label{S:VorticityFormulation} The definition of a weak solution to the 2D Euler equations for initial velocity in $S$ given in \cite{AKLL2014} required that the solutions satisfy the Serfati identity, \refE{SerfatiId} (with $U_\iny \equiv 0$). This requirement was to insure uniqueness of solutions.

The Serfati identity encodes information both about the membership of the velocity field in $S$ and the PDE (the Euler equations) that the velocity field satisfies. The renormalized Biot-Savart law of \refE{RenBSLaw} only encodes the membership of the velocity field in a subspace of $S$ for which the renormalized Biot-Savart law holds without taking a subsequence. It follows from \refTAnd{CharacterizationR2}{ExistenceR2} that we can use the renormalized Biot-Savart law---specifying the value of $U_\iny$---instead of the Serfati identity as our selection criterion to insure uniqueness. This is more satisfying, as it reduces redundancy in the definition of a weak solution, and gives us the vorticity formulation of a weak solution in \refD{ESolEandU}, suitable for insuring both existence and uniqueness. Moreover, this definition is quite close to the usual vorticity formulation of solutions to the 2D Euler equations.

\begin{definition}[Vorticity formulation of a weak solution in $\R^2$]\label{D:ESolEandU}
    Fix $T > 0$ and $U_\iny \in C([0, T])$ with $U_\iny(0) = 0$.
    Let $a$ be any radial cutoff function as in \refD{RadialCutoff}.
    Let $u^0 \in S(\R^2)$ with vorticity
    $\omega^0 = \omega(u^0)$.
    We say that
    $\omega \in L^\iny([0, T] \times \R^2)$
    is a \textit{bounded solution} to the Euler equations
    without forcing having initial velocity $u^0$ and weak velocity at
    infinity $U_\iny$ relative to $u_0$
    if $u(0) = u^0$ and the following hold:
    \begin{enumerate}            
        \item
            $\prt_t \omega + K[\omega] \cdot \grad \omega = 0$
            as distributions on $(0, T) \times \R^2$,
            where $\omega = \omega(u)$;
            
        \item
            the velocity is recovered from the vorticity via 
            \begin{align}\label{e:VortFormRecovery}
                K[\omega](t) &=  u^0  + U_\iny(t)
                + \lim_{R \to \iny} (a_R K) * (\omega(t) - \omega^0),
            \end{align}
            and $K[\omega] \in C([0, T] \times \R^2)$;

        \item
            the vorticity is transported by the flow map for $K[\omega]$.
            
    \end{enumerate}
\end{definition}

A few comments on this definition:

\begin{enumerate}[(a)]
    \item
        By \refT{CharacterizationR2}, \refD{ESolEandU} does not depend
        upon the particular choice of the radial cutoff function $a$.
        
    \item
        Even for a vorticity formulation, we must specify not just the
        bounded initial vorticity but the initial velocity, insisting
        that it too be bounded. This is because there are vorticity
        fields, $\omega^0$, having no corresponding
        bounded velocity, $u^0$, a simple example being
        $\omega^0 \equiv 1$. Moreover, even if a $u^0$ exists it is
        unique only up to an additive constant.
        
    \item
        From $K[\omega] \in C([0, T] \times \R^2)$, the existence
        and uniqueness of a classical flow map follows as in
        \refR{FlowMap}.
        
    \item
        The assumption that the velocity, $K[\omega]$, lie in
        $C([0, T] \times \R^2)$ seems to be necessary, as it does not
        follow from \refE{VortFormRecovery}.
\end{enumerate}

\subsection{Relation to work of Taniuchi, Tashiro, and Yoneda}\label{S:TTYWork}
To construct his solutions to the Euler equations in \cite{Taniuchi2004}, Taniuchi uses a sequence of approximating smooth solutions coming from \cite{Serfati1995B}. In particular, he uses \refE{gradpR2} (for $U_\iny \equiv 0$) to obtain the formula,
\begin{align*}
	u(t_2)
		= u(t_1) - \int_{t_1}^{t_2} \LP (u \cdot \grad u)(t) \, dt,
\end{align*}
where $\LP$ is \textit{formally} the Leray projector, defined in terms of Riesz transforms. This formula plays somewhat the same function that \refE{SerfatiId} plays in \cite{AKLL2014}, and is central in Taniuichi's proof of  existence of bounded (in fact, slightly unbounded) solutions. He does not, however, show that the vorticity is transported by the flow map.

Interestingly, the transport of the vorticity by the flow map is not needed to prove uniqueness of bounded solutions in the full plane. They use the techniques of paradifferential calculus, along the lines of that of Vishik in \cite{VishikBesov}, to obtain continuity with respect to initial data and hence uniqueness.

We show in \refT{ExistenceR2} that, given any $u^0 \in S$ for the full plane, there exists a bounded solution as in \refD{ESol} for which, when $U_\iny \equiv 0$, uniqueness criterion (c) holds. It is shown in \cite{TaniuchiEtAl2010} that such solutions are unique (this result is what our proof of uniqueness criterion (c) was based on). Therefore, the solutions constructed by Taniuchi in \cite{Taniuchi2004}, or at least the subclass of them with bounded vorticity, do, in fact, have their velocity transported by the flow, and so are equivalent to those constructed in \cite{AKLL2014}.

\subsection{Relation to the work of Jun Kato}\label{S:JunKato}

\noindent In \cite{JunKato2003}, Jun Kato studies solutions to the Navier-Stokes equations in all of $\R^n$, $n \ge 2$ when the initial velocity is bounded. We restrict our comments here to the $n = 2$ case, where existence of solutions globally in time holds with\begin{align}\label{e:pRiesz}
    u \in L^\iny([0, T] \times \R^2)
        \text{ and }
    p
        = R_i R_j u^i u^j,
\end{align}
where $R_j = (-\Delta)^{\frac{1}{2}} \prt_j$ is a Riesz transform (\cite{CannonKnightly1970, Knightly1972, Cannone1995, GKM1999, GMS2001}).
Uniqueness was known to hold under the condition that \refE{pRiesz} holds. The uniqueness condition was weakened somewhat in \cite{GKKM2001}, then in \cite{JunKato2003} it was weakened quite a bit further to
\begin{align}\label{e:pRiesz2}
    u \in L^\iny([0, T] \times \R^2)
        \text{ and }
    p
        \in L^1_{loc}([0, T); BMO),
\end{align}
thereby dropping the requirement that the pressure satisfy any particular functional relation.

Kato employs in \cite{JunKato2003} a sequence of approximate Riesz operators, $R^\eps$, converging to the Riesz transform $R$ of \refS{PressureR2} as $\eps \to 0^+$, by cutting off the Green's function for the Laplacian. This same approach could have been taken here, since \refL{ThreeRussians}, which as at the heart of the proof of \refE{pBoundR2}, holds uniformly when using $R^\eps$ in place of $R$. Instead of approximating the Riesz transform used to obtain the pressure, we, in \refS{PressureR2}, approximated the pressure itself. This has the virtue that it can, with substantial additional technical difficulties, be adapted to the exterior of a single obstacle. (We make a few comments on this in \refS{ExteriorDomain}.)

A question that remains open is whether the condition in \refE{pRieszRel} can be dropped as long as \refE{pBMO} holds: this is what is done in \cite{JunKato2003} for the Navier-Stokes equations. What makes this difficult to prove for the Euler equations is that the Leray projector is not bounded in $L^\iny$. For the Navier-Stokes equations, Kato gets around this by taking advantage of properties of the heat kernel. The key estimate, in Lemma 1 of \cite{JunKato2003}, however, blows up like $(\nu t)^{-1/2}$, which prevents the estimate from being adapted for use with the Euler equations.

Finally, we note that the characterization at infinity in \refE{CharacterizationR2} can be extended to solutions to the Navier-Stokes equations with bounded initial velocity \textit{and} vorticity. This is because the analog of the Serfati identity, \refE{SerfatiId}, for the cutoff function, $a_R$, includes only the one additional term,
\begin{align*}
    \nu \int_0^t \Delta_y
			       \pr{(1 - a_R) K^j} * \omega(s) \, ds,
\end{align*}
which vanishes as $R \to \iny$. This allows the argument in the proof of \refP{RenormalizationLemma} to be made without change.

\newcommand{\JJ}{J_\Omega}

\subsection{Exterior to a single obstacle}\label{S:ExteriorDomain}

\noindent It is possible to obtain similar results for the exterior, $\Omega$, to a single, simply connected obstacle having a $C^{2, \al}$ boundary, $\al \in (0, 1)$. We give here a brief account of those results and comment on how they are obtained.

The main result, in analog with \refE{CharacterizationR2}, is that
\begin{align}\label{e:CharacterizationExt}
	\begin{split}
			u(t, x) - u^0(x)
				&= U(t, x)
					+ 
				\lim_{R \to \iny}\int_\Omega a_R(x - y)
				    \JJ(x, y) \omega(y)
				    \, dy, \\
		\grad p(t, x) &= - \prt_t U(t, x) + O(1), \\
		p(t, x) &= - \prt_t \zeta(t, x) + O(\log \abs{x}).
	\end{split}
\end{align}
Here, $\JJ$ is the hydrodynamic Biot-Savart kernel (see \cite{AKLL2014}) and $U$ is a bounded harmonic vector field (that is, divergence-free, curl-free, and tangential to the boundary), which is defined uniquely by its value, $U_\iny$, at infinity and its circulation, $\gamma$, about the boundary. The function, $\gamma$, is the difference in the circulation of $u^0$ from that of $u(t)$. The vector field, $\zeta$, and so the pressure, are multi-valued (unless $\gamma \equiv 0$) with $\grad \zeta = U$. For physically meaningful solutions, we would require that $\gamma \equiv 0$, so that the pressure is single-valued and the circulation is unchanging.

The presence of an obstacle prevents us from transforming the vector field $U$ (or even its value, $U_\iny$, at infinity) away by making a change of reference frame, as we are able to do for the full plane. (Unless we wish to transform the problem to that of a moving obstacle.)

The proof of \refE{CharacterizationExt} parallels that given here for \refE{CharacterizationR2} but is substantially more technical and lengthy for the following reasons:
\begin{enumerate}
    \item
        Formulae involving convolutions in the full plane are
        replaced by integrals over $\Omega$. \refL{Conv}, which
        allowed us to move derivatives back and forth in
        convolutions, must be replaced by integrating by parts,
        which introduces boundary terms that must be controlled.
        This complicates considerably the adaptation of the
        argument in \refS{BSLawR2} to an exterior domain.
        
    \item
        The presence of boundary terms also makes the analog
        of \refL{gammaLemma} for an exterior domain impossible
        to obtain. Instead, we need to strengthen the notion
        of a solution to require that
        \begin{align*}
		    \int_\Omega (\varphi(t) \omega(t, \cdot)
		        - \varphi(0) &\omega^0(\cdot))
			    - \int_0^t \int_\Omega \prt_t \varphi \, \omega \\
			    &- \int_0^t \int_\Omega (\grad \varphi
			        \cdot u) \omega
			    = 0
	    \end{align*}
        for all $\varphi \in C^\iny_C([0, T] \times
	    \ol{\Omega})$, $t \in [0, T]$.
	    
    \item
        The equivalent of \refL{pDecay} is much harder to obtain,
        and involves introducing a Neumann function (Green's
        function of the second kind) to solve for the pressure
        in terms of the velocity. This in turn requires the
        careful control of boundary integrals that do not appear
        for the full plane.
        
    \item
        The estimates on the hydrodynamic Biot-Savart kernel,
        $\JJ$, corresponding to \refP{KStarBounds}
        are considerably harder to obtain than those for
        the Biot-Savart kernel, $K$, for the full plane.
        Fortunately, the needed  estimates were obtained in
        \cite{AKLL2014}.
\end{enumerate}

%
%
\appendix
\section{Some lemmas}\label{A:SomeLemmas}

\begin{lemma}\label{L:Morrey}
	Suppose $u \in S$. Then $u \in LL$ with $\norm{u}_{LL} \le C \norm{u}_S$.
	Moreover, for any bounded domain, $D \subseteq \R^2$,
	\begin{align*}
		\norm{\grad u}_{L^p(D)}
			\le C \abs{D}^{1/p} \frac{p^2}{p - 1} \norm{u}_S.
	\end{align*}
\end{lemma}
\begin{proof}
	See \cite{AKLL2014}.
\end{proof}

Let $\Cal{S}' = \Cal{S}'(\R^2)$ be the space of tempered distributions and $\Cal{E}' = \Cal{E}'(\R^2)$ be the subspace of compactly supported tempered distributions. We make frequent use of the following classical result:

\begin{lemma}\label{L:Conv}
	Suppose that $f \in \Cal{E}'$ and $g \in \Cal{S}'$. Then $f * g = g * f$ lies in $\Cal{S}'$ and
	\begin{align*}
		D^\al(f * g) = D^\al f * g = f * D^\al g
	\end{align*}
	for all multi-indices, $\al$.
\end{lemma}

The following are two integration-by-parts lemmas for low regularity solutions; the first is a standard fact, the second is Theorem I.1.2 of \cite{T2001}.
\begin{lemma}\label{L:Hm1H01}
	Let $U$ be an open subset of $\R^2$.
	If $f$ lies in $H_0^1(U)$ and $v$ lies in $(L^2(U))^2$ then $\dv v$ lies in $H^{-1}(U)$ and
	\begin{align*}
		(\dv v, f)_{H^{-1}(U), H_0^1(U)}
			= - \int_U \grad f \cdot v.
	\end{align*}
\end{lemma}

\begin{lemma}\label{L:TemamIBP}
	Let $U$ be an open subset of $\R^2$ with smooth boundary.
	Let $E(U) = \set{u \in (L^2(U))^2 \colon \dv u \in L^2(U)}$ endowed with the norm,
	$\smallnorm{u}_{E(U)} = \smallnorm{u} + \smallnorm{\dv u}$.
	There exists a continuous trace operator from $E(U)$ to $H^{-1/2}(\prt U)$, which we write
	as $u \mapsto u \cdot \n$, that extends the restriction to the boundary of the normal
	component of $u$ for continuous vector fields.
	Assume that $f$ lies in $H^1(U)$ and $u$ lies in $E(U)$.
	Then
	\begin{align*}
		\int_U u \cdot \grad f + \int_U \dv u \, f
			= (u \cdot \n, f)_{H^{-1/2}(\prt U), H^{1/2}(\prt U)},
	\end{align*}
	where $f$ is the usual trace operator from $H^1(U)$ to $H^{1/2}(U)$ applied to $f$.
\end{lemma}

%
%
\section*{Acknowledgements}

Work on this paper was supported in part by NSF Grants DMS-1212141 and DMS-1009545. The author appreciates helpful conversations with Helena Nussenzveig Lopes and Milton Lopes Filho.



\def\cprime{$'$} \def\polhk#1{\setbox0=\hbox{#1}{\ooalign{\hidewidth
  \lower1.5ex\hbox{`}\hidewidth\crcr\unhbox0}}}


\end{document}